\documentclass[11pt]{article}

\usepackage{epsfig,epsf,fancybox}
\usepackage{amsmath}
\usepackage{mathrsfs}
\usepackage{amssymb}
\usepackage{graphicx}
\usepackage{color}
\usepackage{multirow}
\usepackage{paralist}
\usepackage{verbatim}
\usepackage{galois}
\usepackage{algorithm}
\usepackage{algorithmic}
\usepackage{boxedminipage}
\usepackage{booktabs}
\usepackage{accents}
\usepackage{stmaryrd}

\usepackage{subfig}

\usepackage{natbib}

\usepackage{url}
\usepackage[colorlinks,linkcolor=magenta,citecolor=blue, pagebackref=true,backref=true]{hyperref}
\renewcommand*{\backrefalt}[4]{%
    \ifcase #1 \footnotesize{(Not cited.)}%
    \or        \footnotesize{(Cited on page~#2.)}%
    \else      \footnotesize{(Cited on pages~#2.)}%
    \fi}

\newtheorem{theorem}{Theorem}[section]

\newtheorem{lemma}[theorem]{Lemma}
\newtheorem{proposition}[theorem]{Proposition}

\newtheorem{remark}[theorem]{Remark}

\numberwithin{equation}{section}

\newcommand{\Br}{\mathbb{R}}
\newcommand{\BB}{\mathbb{B}}

\newcommand{\proj}{\mathcal{P}}

\newcommand{\op}{\textnormal{op}}

\newcommand{\argmin}{\mathop{\rm argmin}}

\newcommand{\ECal}{\mathcal{E}}
\newcommand{\GCal}{\mathcal{G}}
\newcommand{\HCal}{\mathcal{H}}

\newcommand{\XCal}{\mathcal{X}}

\newcommand{\dom}{\textnormal{\bf dom}}

\newcommand{\br}{\mathbb{R}}

\newcommand{\ba}{\begin{array}}
\newcommand{\ea}{\end{array}}

\begin{document}


\begin{center}

{\bf{\LARGE{Monotone Inclusions, Acceleration and \\ [.2cm] Closed-Loop Control}}}

\vspace*{.2in}
{\large{ \begin{tabular}{c}
Tianyi Lin$^\diamond$ \and Michael I. Jordan$^{\diamond, \dagger}$ \\
\end{tabular}
}}

\vspace*{.2in}

\begin{tabular}{c}
Department of Electrical Engineering and Computer Sciences$^\diamond$ \\
Department of Statistics$^\dagger$ \\
University of California, Berkeley \\
\end{tabular}

\vspace*{.2in}

\today

\vspace*{.2in}

\begin{abstract}
We propose and analyze a new dynamical system with \textit{a closed-loop control law} in a Hilbert space $\HCal$, aiming to shed light on the acceleration phenomenon for \textit{monotone inclusion} problems, which unifies a broad class of optimization, saddle point and variational inequality (VI) problems under a single framework. Given an operator $A: \HCal \rightrightarrows \HCal$ that is maximal monotone, we propose a closed-loop control system that is governed by the operator $I - (I + \lambda(t)A)^{-1}$, where a feedback law $\lambda(\cdot)$ is tuned by the resolution of the algebraic equation $\lambda(t)\|(I + \lambda(t)A)^{-1}x(t) - x(t)\|^{p-1} = \theta$ for some $\theta > 0$. Our first contribution is to prove the existence and uniqueness of a global solution via the Cauchy-Lipschitz theorem. We present a simple Lyapunov function for establishing the weak convergence of trajectories via the Opial lemma and strong convergence results under additional conditions.  We then prove a global ergodic convergence rate of $O(t^{-(p+1)/2})$ in terms of a gap function and a global pointwise convergence rate of $O(t^{-p/2})$ in terms of a residue function. Local linear convergence is established in terms of a distance function under an error bound condition.  Further, we provide an algorithmic framework based on the implicit discretization of our system in a Euclidean setting, generalizing the large-step HPE framework~\citep{Monteiro-2012-Iteration}. Even though the discrete-time analysis is a simplification and generalization of existing analyses for a bounded domain, it is largely motivated by the aforementioned continuous-time analysis, illustrating the fundamental role that the closed-loop control plays in acceleration in monotone inclusion. A highlight of our analysis is a new result concerning $p^\textnormal{th}$-order tensor algorithms for monotone inclusion problems, complementing the recent analysis for saddle point and VI problems~\citep{Bullins-2022-Higher}. 
\end{abstract}

\end{center}

\section{Introduction}
Monotone inclusion refers to the problem of finding a root of a point-to-set maximal monotone operator $A: \HCal \rightrightarrows \HCal$ (see the definition in~\citet{Rockafellar-1970-Convex}), where $\HCal$ is a real Hilbert space. Formally, we have
\begin{equation}\label{prob:MI}
\textnormal{Find } x \in \HCal \textnormal{ such that } 0 \in Ax. 
\end{equation}
Monotone inclusion is a fundamental problem in applied mathematics, unifying a broad class of optimization, saddle point and variational inequality problems in a single framework. In particular,  the minimization of a convex function $f$ consists in finding $x \in \HCal$ such that $0 \in \partial \Phi(x)$, where $\partial \Phi(\cdot)$---the subdifferential of $\Phi$---is known to be maximal monotone if $\Phi$ is proper, lower semi-continuous and convex. As a further example, letting $A = F+\partial \textnormal{\bf 1}_\XCal$ where $F: \HCal \mapsto \HCal$ is continuous and monotone and $\textbf{1}_\XCal$ is an indicator function of a closed and convex set $\XCal \subseteq \HCal$, the monotone inclusion problem becomes
\begin{equation*}
\textnormal{Finding } x \in \XCal \textnormal{ such that } \langle F(x), y-x\rangle \geq 0 \textnormal{ for all } y \in \XCal. 
\end{equation*}
This is known as the variational inequality (VI) problem~\citep{Facchinei-2007-Finite}, and it also covers many classical problems as special cases~\citep{Karamardian-1972-Complementarity, Kelley-1995-Iterative}. Over several decades, the monotone inclusion problem has found  applications in a wide set of fields, including partial differential equations~\citep{Polyanin-2003-Handbook}, game theory~\citep{Osborne-2004-Introduction}, signal/image processing~\citep{Bose-2003-Digital} and location theory~\citep{Farahani-2009-Facility}; see also~\citet[Section~1.4]{Facchinei-2007-Finite} for additional applications. Recently, the model has begun to see applications in machine learning as an abstraction of saddle point problems, with examples including generative adversarial networks (GANs)~\citep{Goodfellow-2014-Generative}, online learning in games~\citep{Cesa-2006-Prediction}, adversarial learning~\citep{Sinha-2018-Certifiable} and distributed computing~\citep{Shamma-2008-Cooperative}. These applications have made significant demands with respect to computational feasibility, and the design of efficient algorithms for solving monotone inclusions has moved to the fore in the past decade~\citep{Eckstein-2009-General, Briceno-2011-Monotone, Combettes-2013-Systems, Combettes-2018-Asynchronous, Davis-2015-Convergence, Briceno-2018-Forward}.

A simple and basic tool for solving monotone inclusion problems is the celebrated proximal point algorithm (PPA)~\citep{Martinet-1970-Regularisation,Martinet-1972-Determination,Rockafellar-1976-Monotone}.  The idea is to reformulate Eq.~\eqref{prob:MI} as a fixed-point problem given by 
\begin{equation}\label{prob:FP}
\textnormal{Find } x \in \HCal \textnormal{ such that } x - (I + \lambda A)^{-1}x = 0,
\end{equation}
where $\lambda > 0$ is a parameter and $(I + \lambda A)^{-1}$ is the resolvent of index $\lambda$ of $A$. Letting $x_0 \in \HCal$ be an initial point, the PPA scheme is implemented by  
\begin{equation*}
x_{k+1} = (I + \lambda A)^{-1}x_k, \quad \textnormal{for all } k \geq 0. 
\end{equation*}
In the special case of convex optimization, where $A = \partial f$, the convergence rate of PPA is $O(1/k)$ in terms of objective function gap~\citep{Guler-1991-Convergence}.  It has been accelerated to $O(1/k^2)$ by~\citet{Guler-1992-New}. However, this acceleration can not be extended to  monotone inclusion problems in full generality, although extensions have been found under certain conditions (e.g., cocoercivity)~\citep{Alvarez-2001-Inertial,Attouch-2019-Convergence,Attouch-2020-Convergence}. In the context of monotone VIs, the ergodic convergence rate is $O(1/k)$ in terms of a gap function and the pointwise convergence rate is $O(1/\sqrt{k})$ in terms of a residue function~\citep{Facchinei-2007-Finite}. The former rate has matched the lower bound for first-order methods~\citep{Diakonikolas-2020-Halpern} while the latter rate can be improved using new acceleration techniques~\citep{Kim-2021-Accelerated}.  This line of work focuses, however, on first-order algorithms and does not regard acceleration as a general phenomenon to be realized via appeal to high-order smoothness structure of an operator. As noted in a seminal work~\citep{Monteiro-2012-Iteration}, there remains a gap in our understanding of accelerated $p^\textnormal{th}$-order tensor algorithms for monotone inclusion problems, for the case of $p \geq 2$, where the algorithmic design and convergence analysis is much more delicate.

In this paper, we avail ourselves of a continuous-time viewpoint for formulating acceleration in monotone inclusion, making use of a closed-loop control mechanism.  We build on a two-decade trend that exploits the interplay between continuous-time and discrete-time perspectives on dynamical systems for monotone inclusion problems~\citep{Alvarez-1998-Dynamical,Attouch-2001-Second,Alvarez-2002-Second,Attouch-2011-Continuous,Attouch-2012-Second,Mainge-2013-First,Attouch-2013-Global,Abbas-2014-Newton,Attouch-2016-Dynamic,Attouch-2020-Newton,Attouch-2021-Continuous}. As in these papers, our work makes use of Lyapunov functions to transfer asymptotic behavior and rates of convergence between continuous time and discrete time. 

Our point of departure is the following continuous-time problem that incorporates a time-varying function $\lambda(\cdot)$ in place of $\lambda$ in the fixed-point formulation of monotone inclusion in Eq.~\eqref{prob:FP}:
\begin{equation}\label{sys:general}
\dot{x}(t) + x(t) - (I + \lambda(t)A)^{-1}x(t) = 0.
\end{equation}
The time evolution  of $\lambda(\cdot)$ is specified by a closed-loop control law:
\begin{equation}\label{sys:choice-feedback}
\lambda(t)\|\dot{x}(t)\|^{p-1} = \theta,
\end{equation}
where $\theta > 0$ and the order $p \in \{1, 2, \ldots\}$ are parameters. We assume that $x(0) \in \{x \in \HCal \mid 0 \notin Ax\}$.  This is not restrictive since $0 \in Ax(0)$ implies that the monotone inclusion problem has been solved. Throughout the paper, we assume that \textit{$A$ is maximal monotone and $A^{-1}(0) = \{x \in \HCal \mid 0 \in Ax\}$ is a nonempty set}. As we shall see, our main results on the existence and uniqueness of global solutions and the convergence properties of trajectories are valid under this general assumption. Finally, we remark that our control law in Eq.~\eqref{sys:choice-feedback} is a natural generalization of a similar equation in~\citet{Attouch-2016-Dynamic} that models the proximal Newton algorithm specialized to convex optimization.

\paragraph{Contributions.} We first study the closed-loop control system in Eq.~\eqref{sys:general} and~\eqref{sys:choice-feedback} and prove the existence and uniqueness of a global solution via the Cauchy-Lipschitz theorem (see Theorem~\ref{Theorem:Global-Existence-Uniquess}). For $p=1$, we have $\lambda(t) = \theta$ and our system becomes the continuous-time PPA dynamics, indicating that our system extends PPA from first-order monotone inclusion to high-order monotone inclusion.  Further, we provide a Lyapunov function that allows us to establish weak convergence of trajectories via the Opial lemma (see Theorem~\ref{Theorem:Weak-Convergence}) and yield strong convergence results under additional conditions (see Theorem~\ref{Theorem:Strong-Convergence}).  We obtain an ergodic convergence rate of $O(t^{-(p+1)/2})$ in terms of a gap function and a pointwise convergence rate of $O(t^{-p/2})$ in terms of a residue function (see Theorem~\ref{Theorem:Trajectory-Convergence-Rate}). Local linear convergence guarantee is established under an error-bound condition (see Theorem~\ref{Theorem:Trajectory-Convergence-Linear}).  Moreover, we provide an algorithmic framework based on the implicit discretization of our closed-looped control system and remark that it generalizes the large-step HPE framework of~\citet{Monteiro-2012-Iteration}. Our iteration complexity analysis, which is largely motivated by our continuous-time analysis, can be viewed as a simplification and generalization of the analysis in~\citet{Monteiro-2012-Iteration} for bounded domains (see Theorem~\ref{Theorem:CAF-main}). Finally, we combine our algorithmic framework with an approximate tensor subroutine, yielding a suite of accelerated $p^\textnormal{th}$-order tensor algorithms for monotone inclusion problems with $A = F + H$, where $F$ has Lipschitz $(p-1)^\textnormal{th}$-order derivative and $H$ is simple and maximal monotone. A highlight of our analysis is a set of new theoretical results concerning the convergence rate of $p^\textnormal{th}$-order tensor algorithms for monotone inclusion problems, complementing the previous analysis in~\citet{Bullins-2022-Higher}. 

\paragraph{Organization.} In Section~\ref{sec:control-system}, we study the closed-loop control system in Eq.~\eqref{sys:general} and~\eqref{sys:choice-feedback} and prove the global existence and uniqueness results. In Section~\ref{sec:convergence}, we establish the weak convergence of trajectories as well as strong convergence results under additional conditions. We also prove the ergodic and pointwise convergence rates of trajectories in terms of a gap function and a residue function. In Section~\ref{sec:algorithm}, we propose an algorithmic framework based on the implicit discretization of our closed-loop control system in Euclidean setting, and then derive specific accelerated $p^\textnormal{th}$-order tensor algorithms for smooth and monotone inclusion problems. In Section~\ref{sec:conclusions}, we conclude our work with a brief discussion of future directions. All the proofs of lemmas are deferred to the appendix. 

\paragraph{Notation.} We use bold lower-case letters such as $x$ to denote vectors, and upper-case letters such as $X$ to denote tensors.  We let $\HCal$ be a real Hilbert space that is endowed with the scalar product $\langle \cdot, \cdot\rangle$. For a vector $x \in \HCal$, we let $\|x\|$ denote its norm induced by $\langle \cdot, \cdot\rangle$ and let $\BB_\delta(x) = \{x' \in \HCal \mid \|x'-x\| \leq \delta\}$ denote its $\delta$-neighborhood.  For the operator $A: \HCal \rightrightarrows \HCal$, we let $\dom(A) = \{x \in \HCal: Ax \neq \emptyset\}$. If $\HCal = \br^d$ is a real Euclidean space,  $\|x\|$ refers to the $\ell_2$-norm of $x$.  For a tensor $X \in \br^{d_1 \times d_2 \times \ldots \times d_p}$, we define 
\begin{equation*}
X[z^1, \cdots, z^p] = \sum_{1 \leq i_j \leq d_j, 1 \leq j \leq p} \left[X_{i_1, \cdots, i_p}\right]z_{i_1}^1 \cdots z_{i_p}^p, 
\end{equation*}
and denote by $\|X\|_\op = \max_{\|z^i\|=1, 1 \leq j \leq p} X[z^1, \cdots, z^p]$ its operator norm induced by $\|\cdot\|$.  Fix $p \geq 1$, we let $\GCal_L^p(\br^d)$ be a class of maximal monotone single-valued operators $F: \br^d \rightarrow \br^d$ where the $(p-1)^\textnormal{th}$-order Jacobian are $L$-Lipschitz. In other words, $F \in \GCal_L^p(\br^d)$ if $F$ is maximal monotone and $\|D^{(p-1)} F(x') - D^{(p-1)} F(x)\| \leq L\|x' - x\|$ for all $x, x' \in \br^d$ where $D^{(p-1)} F(x)$ is the $(p-1)^\textnormal{th}$-order Jacobian of $F$ at $x \in \br^d$ and $D^{(0)} F = F$ for all $F \in \GCal_L^1(\br^d)$. To be more specific, for $\{z_1, z_2, \ldots, z_p\} \subseteq \br^d$, we have
\begin{equation*}
D^{(p-1)} F(x)[z^1, \cdots, z^p] = \sum_{1 \leq i_1, \ldots, i_p \leq d} \left[\tfrac{\partial F_{i_1}}{\partial x_{i_2} \cdots \partial x_{i_p}}(x)\right] z_{i_1}^1 \cdots z_{i_p}^p. 
\end{equation*}
Given an iteration count $k \geq 1$, the notation $a = O(b(k))$ stands for $a \leq C \cdot b(k)$ where the constant $C > 0$ is independent of $k$. 

\section{The Closed-Loop Control System}\label{sec:control-system}
In this section, we study the closed-loop control system in Eq.~\eqref{sys:general} and Eq.~\eqref{sys:choice-feedback}. We start by analyzing the algebraic equation $\lambda(t)\|(I + \lambda(t)A)^{-1}x(t) - x(t)\|^{p-1} = \theta$ for $\theta \in (0, 1)$. We prove the existence and uniqueness of a local solution by appeal to the Cauchy-Lipschitz theorem and extend the local solution to a global solution using properties of the closed-loop control law $\lambda(\cdot)$. We conclude by discussing other systems in the literature that exemplify our general framework.

\subsection{Algebraic equation}\label{subsec:AE}
We study the algebraic equation,
\begin{equation}\label{Eq:AE-main}
\lambda(t)\|(I + \lambda(t)A)^{-1}x(t) - x(t)\|^{p-1} = \theta \in (0, 1),  
\end{equation}
which links the feedback control law $\lambda(\cdot)$ and the solution trajectory $x(\cdot)$. To streamline the presentation, for the case of $p \geq 2$ we define a function $\varphi: [0, +\infty) \times \HCal \mapsto [0, +\infty)$, such that 
\begin{equation*}
\varphi(\lambda, x) = \lambda^{\frac{1}{p-1}}\|x - (I + \lambda A)^{-1}x\|, \quad \varphi(0, x) = 0. 
\end{equation*}
By the definition of $\varphi$, Eq.~\eqref{Eq:AE-main} is equivalent to $\varphi(\lambda(t), x(t)) = \theta^{1/(p-1)}$. Our first lemma shows that the mapping $x \mapsto \varphi(\lambda, x)$ is Lipschitz continuous for fixed $\lambda > 0$.  We have:
\begin{lemma}\label{Lemma:AE-mapping-Lipschitz}
For $p \geq 2$, we have $|\varphi(\lambda, x_1) - \varphi(\lambda, x_2)| \leq \lambda^{\frac{1}{p-1}}\|x_1 - x_2\|$ for $\forall x_1, x_2 \in \HCal$ and $\forall \lambda > 0$. 
\end{lemma}
The next lemma presents a key property of the mapping $\lambda \mapsto \varphi(\lambda, x)$ for a fixed $x \in \HCal$.  It can be interpreted as a generalization of~\citet[Lemma~4.3]{Monteiro-2012-Iteration} and~\citet[Lemma~1.3]{Attouch-2016-Dynamic} from $p = 2$ to $p \geq 2$.  
\begin{lemma}\label{Lemma:AE-mapping-monotone}
For $p \geq 2$, we have
\begin{equation*}
\left(\tfrac{\lambda_2}{\lambda_1}\right)^{\frac{1}{p-1}}\varphi(\lambda_1, x) \leq \varphi(\lambda_2, x) \leq \left(\tfrac{\lambda_2}{\lambda_1}\right)^{\frac{p}{p-1}}\varphi(\lambda_1, x), 
\end{equation*}
for all $x \in \HCal$ and $0 < \lambda_1 \leq \lambda_2$. In addition, $\varphi(\lambda, x) = 0$ if and only if $0 \in Ax$ for any fixed $\lambda > 0$. 
\end{lemma}
The following proposition provides a property of the mapping $\lambda \mapsto \varphi(\lambda, x)$, for any fixed $x \in \HCal$ satisfying $x \notin A^{-1}(0) = \{x' \in \HCal: 0 \in Ax'\}$. We have:
\begin{proposition}\label{Prop:AE-mapping-all}
Suppose that $p \geq 2$ and $x \notin A^{-1}(0)$ is fixed, the mapping $\varphi(\cdot, x)$ is continuous and strictly increasing. Further, we have $\varphi(0, x) = 0$ and $\varphi(\lambda, x) \rightarrow +\infty$ as $\lambda \rightarrow +\infty$.
\end{proposition}
\begin{proof}
By definition of $\varphi$, we have $\varphi(0, x) = 0$ for any fixed $x \notin A^{-1}(0)$. Since $x \notin A^{-1}(0)$,  Lemma~\ref{Lemma:AE-mapping-monotone} guarantees that $\varphi(\lambda, x) > 0$ for all $\lambda > 0$ and $\varphi(\lambda_1, x) < \varphi(\lambda_2, x)$ for all $0 < \lambda_1 < \lambda_2$.  That is to say,  the mapping $\varphi(\cdot, x)$ is strictly increasing. In addition, we fix $\lambda_1 > 0$ and let $\lambda_2 \rightarrow +\infty$ in Lemma~\ref{Lemma:AE-mapping-monotone}, yielding that $\varphi(\lambda, x) \rightarrow +\infty$ as $\lambda \rightarrow +\infty$ for any fixed $x \notin A^{-1}(0)$. Finally,  we prove the continuity of the mapping $\lambda \mapsto \varphi(\lambda, x)$.  In particular, Lemma~\ref{Lemma:AE-mapping-monotone} implies that $\varphi(\lambda, x) \leq \lambda^{p/(p-1)} \varphi(1, x)$ for any fixed $\lambda \in (0, 1]$.  This together with the definition of $\varphi$ implies that 
\begin{equation*}
0 \leq \limsup_{\lambda \rightarrow 0^+} \varphi(\lambda, x) \leq \lim_{\lambda \rightarrow 0^+} \lambda^{\frac{p}{p-1}} \varphi(1, x) = 0, 
\end{equation*}
which implies the continuity of the mapping $\lambda \mapsto \varphi(\lambda, x)$ at $\lambda = 0$. Left continuity and right continuity of $\lambda \mapsto \varphi(\lambda, x)$ at $\lambda > 0$ follow from the first and the second inequality in Lemma~\ref{Lemma:AE-mapping-monotone}. 
\end{proof}
In view of Proposition~\ref{Prop:AE-mapping-all}, for any fixed $x \notin A^{-1}(0)$, there exists a unique $\lambda > 0$ so that $\varphi(\lambda, x) = \theta^{1/(p-1)}$ for some $\theta \in (0, 1)$. We accordingly define $\Omega \subseteq \HCal$ and the mapping $\Lambda_\theta: \Omega \mapsto (0, \infty)$ as follows:
\begin{equation}\label{def:AE-mapping}
\Omega = \HCal \setminus A^{-1}(0) \doteq \{x \in \HCal: 0 \notin Ax\}, \qquad \Lambda_\theta(x) = (\varphi(\cdot, x))^{-1}(\theta^{1/(p-1)}). 
\end{equation}
Since $A$ is maximal monotone, we have that $A^{-1}(0)$ is closed and thus $\Omega$ is open. Note that this simple fact is crucial to the subsequent analysis of the existence and uniqueness of a local solution. 

\subsection{Existence and uniqueness of a local solution}\label{subsec:LEU}
We prove the existence and uniqueness of a local solution of the closed-loop control system in Eq.~\eqref{sys:general} and Eq.~\eqref{sys:choice-feedback} by appeal to the Cauchy-Lipschitz theorem. The system considered in this paper can be written in the following form:
\begin{equation*}
\left\{\begin{array}{ll}
& \dot{x}(t) + x(t) - (I + \lambda(t)A)^{-1}x(t) = 0, \\
& \lambda(t)\|(I + \lambda(t)A)^{-1}x(t) - x(t)\|^{p-1} = \theta, \\ 
& x(0) = x_0 \in \Omega. 
\end{array}\right.  
\end{equation*}
Using the mapping $\Lambda_\theta: \Omega \mapsto (0, \infty)$ (see Eq.~\eqref{def:AE-mapping}), this system can be expressed as an autonomous system. Indeed, we have
\begin{equation*}
\lambda(t) = \Lambda_\theta(x(t)) \Longleftrightarrow \lambda(t)\|(I + \lambda(t)A)^{-1}x(t) - x(t)\|^{p-1} = \theta.  
\end{equation*}
Putting these pieces together, we arrive at an autonomous system in the compact form of 
\begin{equation}\label{sys:autonomous}
\dot{x}(t) = F(x(t)), \quad x(0) = x_0 \in \Omega,
\end{equation}
where the vector field $F: \Omega \mapsto \HCal$ is given by
\begin{equation}\label{sys:vector-field}
F(x) = (I + \Lambda_\theta(x)A)^{-1}x - x. 
\end{equation}
Our main theorem on the existence and uniqueness of a local solution is summarized as follows. 
\begin{theorem}\label{Theorem:Local-Existence-Uniquess} 
There exists $t_0 > 0$ such that the autonomous system in Eq.~\eqref{sys:autonomous} and Eq.~\eqref{sys:vector-field} has a unique solution $x: [0, t_0] \mapsto \HCal$. Equivalently,  the closed-loop control system in Eq.~\eqref{sys:general} and Eq.~\eqref{sys:choice-feedback} has a unique solution, $(x, \lambda): [0, t_0] \mapsto \HCal \times (0, +\infty)$.  In addition, $x(\cdot)$ is continuously differentiable and $\lambda(\cdot)$ is locally Lipschitz continuous. 
\end{theorem}
A standard approach for proving the existence and uniqueness of a local solution is via appeal to the Cauchy-Lipschitz theorem~\citep[Theorem I.3.1]{Coddington-1955-Theory}. This theorem requires that the vector field $F(\cdot)$ is Lipschitz continuous, which is not immediate in our case due to the appearance of $(I + \Lambda_\theta(x)A)^{-1}$. In order to avail ourselves directly of the Cauchy-Lipschitz theorem, the first step is to study the properties of the function $\Lambda_\theta(x)$.  We have the following lemma. 
\begin{lemma}\label{Lemma:control-Lipschitz}
Suppose that $p \geq 2$ and the function $\Gamma_\theta: \HCal \mapsto (0, +\infty)$ is given by 
\begin{equation*}
\Gamma_\theta(x) = \left(\inf\{\alpha > 0: \left\|x - (I + \alpha^{-1}A)^{-1}x\right\| \leq \alpha^{\frac{1}{p-1}} \theta^{\frac{1}{p-1}}\}\right)^{\frac{1}{p-1}}. 
\end{equation*}
Then, we have
\begin{equation*}
\Gamma_\theta(x) = \left\{
\begin{array}{ll}
\left(\tfrac{1}{\Lambda_\theta(x)}\right)^{\frac{1}{p-1}}, & \textnormal{if } x \in \Omega, \\
0, & \textnormal{otherwise}, 
\end{array}
\right. 
\end{equation*}
and $\Gamma_\theta$ itself is Lipschitz continuous with a constant $\theta^{-1/(p-1)} > 0$. 
\end{lemma}
\paragraph{Proof of Theorem~\ref{Theorem:Local-Existence-Uniquess}:} For simplicity,  we define $A_\theta = I - (I + \theta A)^{-1}$.  In what follows,  we first prove that $F: \Omega \mapsto \HCal$ defined in Eq.~\eqref{sys:vector-field} is locally Lipschitz continuous case by case. 

\textit{Case of $p=1$:} The algebraic equation in Eq.~\eqref{sys:choice-feedback} implies that $\lambda(\cdot)$ is a constant function such that $\lambda(t) \equiv \theta$ for all $t \geq 0$.  Then, by the definition of $F$ in Eq.~\eqref{sys:vector-field}, we have
\begin{equation*}
F(x) = (I + \theta A)^{-1}x - x.  
\end{equation*}
By the definition of $A_\lambda$, we have $\|F(x_1) - F(x_2)\| = \|A_\theta x_1 - A_\theta x_2\|$. It is straightforward to derive that $A_\theta$ is 1-Lipschitz continuous (see the proof of Lemma~\ref{Lemma:AE-mapping-Lipschitz}). Putting these pieces together yields the desired result. 

\textit{Case of $p \geq 2$:} Taking $x_0 \in \Omega$ and $0 < \delta < (\frac{\theta}{\Lambda_\theta(x_0)})^{1/(p-1)}$, we have $B_\delta(x_0, \delta) \subseteq \Omega$ since $\Omega$ is open. For any $x \in B_\delta(x_0, \delta)$,  Lemma~\ref{Lemma:control-Lipschitz} implies
\begin{equation*}
\left|\left(\tfrac{1}{\Lambda_\theta(x)}\right)^{\frac{1}{p-1}} - \left(\tfrac{1}{\Lambda_\theta(x_0)}\right)^{\frac{1}{p-1}}\right| = |\Gamma_\theta(x) - \Gamma_\theta(x_0)| \leq \theta^{-\frac{1}{p-1}}\|x - x_0\| \leq \delta\theta^{-\frac{1}{p-1}}. 
\end{equation*}
In view of the choice of $\delta > 0$ and the definition of $\lambda_0$,  we have
\begin{equation}\label{def:control_bound}
0 < \left(\tfrac{1}{\Lambda_\theta(x_0)}\right)^{\frac{1}{p-1}} - \left(\tfrac{\delta^{p-1}}{\theta}\right)^{\frac{1}{p-1}} \leq \left(\tfrac{1}{\Lambda_\theta(x)}\right)^{\frac{1}{p-1}} = \Gamma_\theta(x) \leq \left(\tfrac{1}{\Lambda_\theta(x_0)}\right)^{\frac{1}{p-1}} + \left(\tfrac{\delta^{p-1}}{\theta}\right)^{\frac{1}{p-1}}. 
\end{equation}
Taking $x_1, x_2 \in B_\delta(x_0, \delta) \subseteq \Omega$, we let $\lambda_1 = \Lambda_\theta(x_1)$ and $\lambda_2 = \Lambda_\theta(x_2)$.  Then, we have
\begin{equation}\label{inequality:local-existence-first}
\|F(x_1) - F(x_2)\| = \|A_{\lambda_1} x_1 - A_{\lambda_2} x_2\| \leq \|A_{\lambda_1} x_1 - A_{\lambda_1} x_2\| + \|A_{\lambda_1} x_2 - A_{\lambda_2} x_2\|. 
\end{equation}
By the definition of $\lambda_1$, we obtain that $\lambda_1 > 0$ and $A_{\lambda_1}$ is 1-Lipschitz continuous. This implies
\begin{equation}\label{inequality:local-existence-second}
\|A_{\lambda_1} x_1 - A_{\lambda_1} x_2\| \leq \|x_1 - x_2\|. 
\end{equation}
Further, we obtain from the proof of~\citet[Lemma~A.4]{Attouch-2019-Convergence} that 
\begin{equation*}
\|A_{\lambda_1} x_2 - A_{\lambda_2} x_2\| \leq \left|1 - \tfrac{\lambda_1}{\lambda_2}\right|\|x_2 - (I + \lambda_2 A)^{-1}x_2)\|. 
\end{equation*}
This together with the definition of $\lambda_2$ and $\Lambda_\theta(\cdot)$ yields
\begin{equation}\label{inequality:local-existence-third}
\|A_{\lambda_1} x_2 - A_{\lambda_2} x_2\| \leq \left|1 - \tfrac{\lambda_1}{\lambda_2}\right| \left(\tfrac{\theta}{\lambda_2}\right)^{\frac{1}{p-1}}. 
\end{equation}
Plugging Eq.~\eqref{inequality:local-existence-second} and Eq.~\eqref{inequality:local-existence-third} into Eq.~\eqref{inequality:local-existence-first} and using the definition of $\lambda_1$ and $\lambda_2$, we have
\begin{equation}\label{inequality:local-existence-fourth}
\|F(x_1) - F(x_2)\| \leq \|x_1 - x_2\| + \left|1 - \tfrac{\Lambda_\theta(x_1)}{\Lambda_\theta(x_2)}\right| \left(\tfrac{\theta}{\Lambda_\theta(x_2)}\right)^{\frac{1}{p-1}}. 
\end{equation}
Using $\Gamma_\theta(x) = (\frac{1}{\Lambda_\theta(x)})^{\frac{1}{p-1}}$ for all $x \in \Omega$ and the Lipschitz continuity of $\Gamma_\theta$ (cf. Lemma~\ref{Lemma:control-Lipschitz}), we have
\begin{eqnarray*}
\lefteqn{\left|1 - \tfrac{\Lambda_\theta(x_1)}{\Lambda_\theta(x_2)}\right| \left(\tfrac{1}{\Lambda_\theta(x_2)}\right)^{\frac{1}{p-1}} = \tfrac{\Gamma_\theta(x_2)}{(\Gamma_\theta(x_1))^{p-1}}\left|(\Gamma_\theta(x_1))^{p-1} - (\Gamma_\theta(x_2))^{p-1}\right|} \\ 
& = & \tfrac{\Gamma_\theta(x_2)}{(\Gamma_\theta(x_1))^{p-1}}\left|\Gamma_\theta(x_1) - \Gamma_\theta(x_2)\right|\left(\sum_{i=1}^{p-1} (\Gamma_\theta(x_1))^{p-1-i}(\Gamma_\theta(x_2))^{i-1}\right) \\ 
& \leq & \tfrac{\Gamma_\theta(x_2)}{(\Gamma_\theta(x_1))^{p-1}}\left(\sum_{i=1}^{p-1} (\Gamma_\theta(x_1))^{p-1-i}(\Gamma_\theta(x_2))^{i-1}\right)\theta^{-\frac{1}{p-1}}\|x_1 - x_2\|. 
\end{eqnarray*}
Plugging this inequality into Eq.~\eqref{inequality:local-existence-fourth} yields 
\begin{equation*}
\|F(x_1) - F(x_2)\| \leq \left(1 + \tfrac{\Gamma_\theta(x_2)}{(\Gamma_\theta(x_1))^{p-1}}\left(\sum_{i=1}^{p-1} (\Gamma_\theta(x_1))^{p-1-i}(\Gamma_\theta(x_2))^{i-1}\right)\right)\|x_1 - x_2\|. 
\end{equation*}
Since $x_1, x_2 \in B_\delta(x_0, \delta)$, Eq.~\eqref{def:control_bound} implies
\begin{equation*}
0 < \left(\tfrac{1}{\Lambda_\theta(x_0)}\right)^{\frac{1}{p-1}} - \left(\tfrac{\delta^{p-1}}{\theta}\right)^{\frac{1}{p-1}} \leq \Gamma_\theta(x_i) \leq \left(\tfrac{1}{\Lambda_\theta(x_0)}\right)^{\frac{1}{p-1}} + \left(\tfrac{\delta^{p-1}}{\theta}\right)^{\frac{1}{p-1}}, \quad \textnormal{for all } i = 1, 2. 
\end{equation*}
Therefore, we conclude that 
\begin{equation*}
\|F(x_1) - F(x_2)\| \leq C\|x_1 - x_2\|, 
\end{equation*}
where $C > 0$ is a constant that is independent of the choice of $x_1$ and $x_2$ but only depends on the value of $\delta$, $\theta$, $p$ and $\Lambda_\theta(x_0)$.  This proves the claim. 

We are ready to prove main results.  Indeed, by the Cauchy-Lipschitz theorem (local version), for any $x_0 \in \Omega$, there exists a unique local solution $x: [0, t_1] \mapsto \HCal$ of the autonomous system in Eq.~\eqref{sys:autonomous} and Eq.~\eqref{sys:vector-field} for some $t_1 > 0$.  Thus, there exists a unique local solution, $(x, \lambda): [0, t_1] \mapsto \HCal \times (0, +\infty)$, of the closed-loop control system in Eq.~\eqref{sys:general} and Eq.~\eqref{sys:choice-feedback} with $\lambda(t) = \Lambda_\theta(x(t))$.  By Cauchy-Lipschitz theorem, we have that $x(\cdot)$ is continuously differentiable and $x(t) \in \Omega$ for all $t \in [0, t_1]$. For the case of $p=1$, Eq.~\eqref{sys:choice-feedback} implies that $\lambda(\cdot)$ is a constant function and thus locally Lipschitz continuous.  For the case of $p \geq 2$, Lemma~\ref{Lemma:control-Lipschitz} together with $x(t) \in \Omega$ for all $t \in [0, t_1]$ implies
\begin{equation*}
\lambda(t) = \Lambda_\theta(x(t)) = \left(\tfrac{1}{\Gamma_\theta(x(t))}\right)^{\frac{1}{p-1}} \quad \textnormal{for all } t \in [0, t_1]. 
\end{equation*}
Since $\Gamma_\theta(x)$ is Lipschitz continuous in $x$, we obtain that $\lambda(\cdot)$ is Lipschitz continuous on $[0, t_2]$ for some sufficiently small $t_2 > 0$.  Then, by taking $t_0 = \min\{t_1, t_2\} > 0$, we achieve the desired results.  This completes the proof. 

\subsection{Existence and uniqueness of a global solution}
Our main theorem on the existence and uniqueness of a global solution is summarized as follows. 
\begin{theorem}\label{Theorem:Global-Existence-Uniquess} 
The closed-loop control system in Eq.~\eqref{sys:general} and Eq.~\eqref{sys:choice-feedback} has a unique global solution, $(x, \lambda): [0, +\infty) \mapsto \HCal \times (0, +\infty)$. Moreover, $x(\cdot)$ is continuously differentiable and $\lambda(\cdot)$ is locally Lipschitz continuous. If $p \geq 2$, we have 
\begin{equation*}
\|x(t) - (I + \lambda(t)A)^{-1}x(t)\| \geq \|x(0) - (I + \lambda(0) A)^{-1}x(0)\|e^{-t}, \quad \textnormal{for all } t \geq 0. 
\end{equation*}
\end{theorem}
\begin{remark}
Theorem~\ref{Theorem:Global-Existence-Uniquess} demonstrates that $x(t) - (I + \lambda(t)A)^{-1}x(t) \neq 0$ for all $t \geq 0$.  After some straightforward calculations, it is clear that the aforementioned argument is equivalent to the assertion that the orbit $x(\cdot)$ stays in $\Omega$. In other words, if $x_0 \in \Omega$, our closed-loop control system in Eq.~\eqref{sys:general} and Eq.~\eqref{sys:choice-feedback} is not stabilized in finite time, which helps clarify the asymptotic convergence behavior of many discrete-time algorithms to a solution of monotone inclusion problems (see~\citet{Monteiro-2010-Complexity, Monteiro-2012-Iteration} for examples). 
\end{remark}
\begin{remark}
The feedback law $\lambda(\cdot)$, which we will show satisfies $\lambda(t) \rightarrow +\infty$ as $t \rightarrow +\infty$, links to $\|\dot{x}(\cdot)\| = \|x(\cdot) - (I + \lambda(\cdot)A)^{-1}x(\cdot)\|$ via Eq.~\eqref{sys:choice-feedback}. Intuitively, if $\lambda(\cdot)$ changes dramatically, we can not globalize a local solution using classical arguments.  In the Levenberg-Marquardt regularized systems,~\citet{Attouch-2011-Continuous} resolved this issue by assuming that $\lambda(\cdot)$ is absolutely continuous on any finite bounded interval and proving that $\lambda(t) \leq \lambda(0)e^{ct}$ holds true for some constant $c > 0$. However, $\lambda(\cdot)$ is not a given datum in our closed-loop control system but an emergent component of the evolution dynamics. As such, it is preferable to prove that $\lambda(t) \leq \lambda(0)e^{ct}$ hold true without imposing any condition, as done in the works~\citep{Attouch-2013-Global, Attouch-2016-Dynamic}. Recently,~\citet{Lin-2022-Control} have studied a closed-loop control system which characterized accelerated $p^\textnormal{th}$-order tensor algorithms for convex optimization and established the global existence and uniqueness results under the condition used in~\citet{Attouch-2011-Continuous}. They also clarified why this condition is necessary and considered it an open problem to remove it.  In the subsequent analysis,  we prove that $|\dot{\lambda}(t)| \leq (p-1)\lambda(t)$ holds for our system in Eq.~\eqref{sys:general} and Eq.~\eqref{sys:choice-feedback} without imposing any condition, demonstrating that acceleration in monotone inclusion problems is  intrinsically different from that in convex optimization.  
\end{remark}

We provide two lemmas that characterize some further properties of the feedback law $\lambda(\cdot)$. 
\begin{lemma}\label{Lemma:control-first}
Suppose that $(x, \lambda): [0, t_0] \mapsto \HCal \times (0, +\infty)$ is a solution of the closed-loop control system in Eq.~\eqref{sys:general} and Eq.~\eqref{sys:choice-feedback}. Then,  we have $|\dot{\lambda}(t)| \leq (p-1)\lambda(t)$ for almost all $t \in [0, t_0]$. 
\end{lemma}
\begin{lemma}\label{Lemma:control-second}
Suppose that $(x, \lambda): [0, t_0] \mapsto \HCal \times (0, +\infty)$ is a solution of the closed-loop control system in Eq.~\eqref{sys:general} and Eq.~\eqref{sys:choice-feedback}. Then,  we have that $\lambda(\cdot)$ is nondecreasing. 
\end{lemma}
\paragraph{Proof of Theorem~\ref{Theorem:Global-Existence-Uniquess}:} We are ready to prove our main result on the existence and uniqueness of a global solution. In particular,  for the case of $p=1$, it is clear that $\lambda(t) = \theta$ is a constant function and the vector field $F: \Omega \mapsto \HCal$ is in fact global Lipschitz continuous (see the proof of Theorem~\ref{Theorem:Local-Existence-Uniquess}). Thus,  by the Cauchy-Lipschitz theorem (global version), we achieve the desired result. 

For the case of $p \geq 2$, let us consider a maximal solution of the closed-loop control system in Eq.~\eqref{sys:general} and Eq.~\eqref{sys:choice-feedback} as follows,  
\begin{equation*}
(x, \lambda): [0, T_{\max}) \mapsto \Omega \times (0, +\infty).  
\end{equation*}
Using the existence and uniqueness of a local solution (see Theorem~\ref{Theorem:Local-Existence-Uniquess}) and a classical argument,  we obtain that the aforementioned maximal solution must exist.  Further, by using Lemma~\ref{Lemma:control-first} and~\ref{Lemma:control-second}, we obtain that $\lambda(\cdot)$ is nondecreasing with $0 \leq \dot{\lambda}(t) \leq (p-1)\lambda(t)$ for almost all $t \in [0, T_{\max})$. 

It remains to show that the maximal solution is a global solution; that is, $T_{\max} = +\infty$. Indeed, the property of $\lambda$ guarantees
\begin{equation}\label{inequality:global-existence-first}
0 < \lambda(0) \leq \lambda(t) \leq \lambda(0)e^{(p-1)t}. 
\end{equation}
If $T_{\max} < +\infty$,  this inequality implies that $\lambda(t) \leq \lambda(0)e^{(p-1)T_{\max}}$ for all $t \in [0, T_{\max}]$. This together with the fact that $\lambda(\cdot)$ is nondecreasing on $[0, T_{\max})$ implies that $\bar{\lambda} = \lim_{t \rightarrow T_{\max}} \lambda(t)$ exists and is finite and strictly positive.  Using Eq.~\eqref{sys:general} and Eq.~\eqref{sys:choice-feedback}, we have
\begin{equation}\label{inequality:global-existence-second}
\|\dot{x}(t)\| = \|(I + \lambda(t)A)^{-1}x(t) - x(t)\| = \left(\tfrac{\theta}{\lambda(t)}\right)^{\frac{1}{p-1}}. 
\end{equation}
Combining Eq.~\eqref{inequality:global-existence-first} and Eq.~\eqref{inequality:global-existence-second} implies that $\|\dot{x}(\cdot)\|$ is bounded on $[0, T_{\max})$. Thus, $x(\cdot)$ is Lipschitz continuous on $[0, T_{\max})$ and this implies that $\bar{x} = \lim_{t \rightarrow T_{\max}} x(t)$ exists. We claim that $\bar{x} \in \Omega$. Indeed, the function $g(\lambda, x) = \|(I + \lambda A)^{-1}x - x\|$ is continuous in $(\lambda, x)$. Since $\lambda(\cdot)$ and $x(\cdot)$ are continuous on $[0,T_{\max}]$, we have $\|(I + \lambda(t)A)^{-1}x(t) - x(t)\| \rightarrow \|(I + \bar{\lambda}A)^{-1}\bar{x} - \bar{x}\|$ as $t \rightarrow T_{\max}$. Then, Eq.~\eqref{inequality:global-existence-second} implies
\begin{equation*}
\|(I + \bar{\lambda}A)^{-1}\bar{x} - \bar{x}\| = \lim_{t \rightarrow T_{\max}} \|(I + \lambda(t)A)^{-1}x(t) - x(t)\| = \lim_{t \rightarrow T_{\max}} \left(\tfrac{\theta}{\lambda(t)}\right)^{\frac{1}{p-1}} = \left(\tfrac{\theta}{\bar{\lambda}}\right)^{\frac{1}{p-1}} > 0. 
\end{equation*}
By appeal to Theorem~\ref{Theorem:Local-Existence-Uniquess} with an initial point $\bar{x}$, we can then extend the solution to a strictly larger interval which contradicts the maximality of the aforementioned solution. 

Using Eq.~\eqref{inequality:global-existence-second} again, we have
\begin{equation*}
\|x(t) - (I + \lambda(t)A)^{-1}x(t)\| = \left(\tfrac{\lambda(0)}{\lambda(t)}\right)^{\frac{1}{p-1}}\|x(0) - (I + \lambda(0)A)^{-1}x(0)\|.
\end{equation*}
Further, it is clear that Eq.~\eqref{inequality:global-existence-first} holds true for all $t \in [0, +\infty)$. That is to say, we have $\frac{\lambda(0)}{\lambda(t)} \geq e^{-(p-1)t}$.  Putting these pieces together yields
\begin{equation*}
\|x(t) - (I + \lambda(t)A)^{-1}x(t)\| \geq e^{-t}\|x(0) - (I + \lambda(0)A)^{-1}x(0)\|,
\end{equation*}
which completes the proof. 

\subsection{Discussion}\label{subsec:Discussion}
We compare the system in Eq.~\eqref{sys:general} and Eq.~\eqref{sys:choice-feedback} to other systems for convex optimization and monotone inclusion. We also give an overview of the closed-loop control approach and the continuous-time interpretation of high-order tensor algorithms. 

\paragraph{Existing systems for optimization and inclusion problems.} In the context of optimization with a convex potential function $\Phi: \HCal \mapsto \br$, ~\citet{Polyak-1964-Some} was the first to use inertial dynamics to accelerate gradient methods.  However, the convergence rate of $O(1/t)$ he obtained is not better than the steepest descent method. A decisive step to obtain a faster convergence rate was taken by~\citet{Su-2016-Differential} who considered using \textit{asymptotically vanishing damping} for modeling Nesterov's acceleration~\citep{Nesterov-1983-Method, Guler-1992-New}, triggering a productive line of research on the dynamical systems foundations of accelerated first-order algorithms~\citep{Attouch-2016-Rate,Attouch-2017-Asymptotic,Attouch-2018-Fast, Diakonikolas-2019-Approximate,Apidopoulos-2020-Convergence, Muehlebach-2021-Optimization}. Another important ingredient for obtaining acceleration is so-called Hessian-driven damping~\citep{Alvarez-2002-Second, Attouch-2016-Fast, Lin-2022-Control, Attouch-2022-Fast, Attouch-2022-First} which originated from a variational characterization of general regularization optimization algorithms~\citep{Alvarez-1998-Dynamical}. This involved the study of Newton and Levenberg-Marquardt regularized systems as follows: 
\begin{align*}
\textbf{(Newton)} \qquad & \ddot{x}(t) + \nabla^2\Phi(x(t))\dot{x}(t) + \nabla\Phi(x(t)) = 0, \\
\textbf{(Levenberg-Marquardt)} \qquad & \lambda(t)\dot{x}(t) + \nabla^2\Phi(x(t))\dot{x}(t) + \nabla\Phi(x(t)) = 0. 
\end{align*}
These systems are well defined and admit robust asymptotic behavior~\citep{Attouch-2001-Second,Attouch-2011-Continuous}. Based on this work, ~\citet{Alvarez-2002-Second} distinguished Hessian-driven damping from continuous-time Newton dynamics and~\citet{Attouch-2016-Fast} interpreted Nesterov's acceleration in the forward-backward algorithms by combining Hessian-driven damping with asymptotically vanishing damping. The resulting dynamics takes the following general form:
\begin{equation}\label{sys:general-potential}
\ddot{x}(t) + \alpha(t)\dot{x}(t) + \beta(t)\nabla^2\Phi(x(t))\dot{x}(t) + b(t)\nabla\Phi(x(t)) = 0. 
\end{equation}    
Further work in this vein appeared in~\citet{Shi-2022-Understanding}, where Nesterov's acceleration was interpreted via multiscale limits that distinguish it from heavy ball method, and~\citet{Attouch-2022-Fast}, where time scaling was introduced with Hessian-driven damping. Unfortunately, none of the above approaches are suitable for deriving optimal accelerated versions of high-order tensor algorithms in convex smooth optimization~\citep{Monteiro-2013-Accelerated, Gasnikov-2019-Near}. Recently, \citet{Lin-2022-Control} provided an initial foray into analyzing the continuous-time dynamics of high-order tensor algorithms using the system in Eq.~\eqref{sys:general-potential} in which the tuning of $(\alpha(\cdot), \beta(\cdot), b(\cdot))$ is done in a closed loop by resolution of the algebraic equation. Their approach gives a systematic way to derive discrete-time optimal high-order tensor algorithms, further simplifying and generalizing the existing analysis in~\citet{Monteiro-2013-Accelerated} via appeal to the construction of a unified discrete-time Lyapunov function. 

The extension of the continuous-time dynamics and Lyapunov analysis from convex optimization to monotone inclusion problems has been pursued during the last two decades~\citep{Alvarez-2001-Inertial,Attouch-2011-Asymptotic,Attouch-2011-Continuous,Mainge-2013-First,Attouch-2013-Global,Attouch-2016-Dynamic,Abbas-2014-Newton,Bot-2016-Second,Attouch-2019-Convergence,Attouch-2018-Convergence,Attouch-2020-Convergence,Attouch-2020-Newton,Attouch-2021-Continuous}. In particular,~\citet{Attouch-2011-Continuous} considered a generalization of Levenberg-Marquardt regularized systems for monotone inclusion problems as follows, 
\begin{equation*}
\left\{\begin{array}{ll}
& v(t) \in Ax(t), \\
& \lambda(t)\dot{x}(t) + \dot{v}(t) + v(t) = 0. 
\end{array}\right.  
\end{equation*}
This system yields weak convergence to $A^{-1}(0)$ under a certain condition on $\lambda(\cdot)$. Subsequent work has obtained convergence rates for various first-order algorithms obtained by the implicit discretization of this system or its variants~\citep{Attouch-2013-Global,Abbas-2014-Newton,Attouch-2016-Dynamic}. 

Under the assumption that $A$ is point-to-point and \textit{cocoercive}, inertial systems taking the following form have been considered in the literature~\citep{Alvarez-2001-Inertial,Attouch-2011-Asymptotic,Mainge-2013-First,Bot-2016-Second}: 
\begin{equation*}
\ddot{x}(t) + \alpha\dot{x}(t) + A(x(t)) = 0. 
\end{equation*}
It is worth mentioning that cocoercivity is necessary for guaranteeing weak asymptotic stabilization, and a fast convergence rate. For $\lambda > 0$, the operator $A_\lambda = \frac{1}{\lambda}(I - (I + \lambda A)^{-1})$ is $\lambda$-cocoercive and $A^{-1}(0) = A_\lambda^{-1}(0)$. This motivates us to study the following inertial system: 
\begin{equation*}
\ddot{x}(t) + \alpha\dot{x}(t) + A_\lambda(x(t)) = 0.
\end{equation*}
In the quest for faster convergence,~\citet{Attouch-2019-Convergence} combined this system with asymptotically vanishing damping and a time-dependent regularizing parameter $\lambda(\cdot)$: 
\begin{equation*}
\ddot{x}(t) + \alpha(t)\dot{x}(t) + A_{\lambda(t)}(x(t)) = 0. 
\end{equation*}
The discretization of these dynamics gives the relaxed inertial proximal algorithm~\citep{Attouch-2018-Convergence,Attouch-2019-Time,Attouch-2019-Convergence,Attouch-2020-Convergence}.  Recently,~\citet{Attouch-2020-Newton,Attouch-2021-Continuous} have proposed to study Newton-like inertial dynamics which generalizes the system in Eq.~\eqref{sys:general-potential} from convex optimization to monotone inclusion problems. The resulting dynamics takes the following general form:
\begin{equation}\label{sys:general-inclusion}
\ddot{x}(t) + \alpha(t)\dot{x}(t) + \beta(t)\tfrac{d}{dt}(A_{\lambda(t)}x(t)) + b(t)A_{\lambda(t)}x(t) = 0. 
\end{equation}    
The introduction of Newton-like correction term $\frac{d}{dt}(A_{\lambda(t)}x(t))$---the generalization of the Hessian-driven damping -- provides a well-posed system for which we can derive the weak convergence of trajectories to $A^{-1}(0)$. The convergence rates have also been obtained in both continuous-time and discrete-time cases using the metric $\|A_{\lambda(t)}x(t)\|$(see~\citet{Attouch-2019-Convergence} and~\citet{Attouch-2020-Newton,Attouch-2021-Continuous}). In contrast, we investigate different dynamical systems and derive different rates in terms of two more intuitive metrics---a gap function and a residue function.

All of the aforementioned dynamical systems study first-order algorithms for monotone inclusion problems and do not aim to capture the acceleration that may be obtainable from high-order smoothness structures. The only exception that we are aware of is~\citet{Attouch-2016-Dynamic} who proposed a proximal Newton method for solving monotone inclusion problems but conducted the convergence rate estimation when an operator is the subdifferential of a convex function. Meanwhile, \citet{Monteiro-2012-Iteration} and~\citet{Bullins-2022-Higher} have demonstrated that high-order tensor algorithms can achieve faster convergence rate than first-order algorithms, but their derivations depend heavily on case-specific algebra.  As such, there remains a gap in our understanding; in particular, we are missing a continuous-time perspective on acceleration in monotone inclusion. 

\paragraph{Closed-loop control systems.}  Closed-loop control systems have been studied in the context of convex optimization~\citep{Lin-2022-Control,Attouch-2022-Loop} and monotone inclusion~\citep{Attouch-2013-Global,Attouch-2016-Dynamic}. Even though~\citet{Attouch-2013-Global,Attouch-2016-Dynamic} closely resembles our work, some differences exist. In particular, their convergence analysis of Newton-type methods targets the solution of convex optimization rather than monotone inclusion problems.  Our focus, on the other hand, is to link closed-loop control with acceleration in monotone inclusion, especially when $p > 2$.  From a technical viewpoint, the construction of the gap function and the convergence rate estimation that we provide do not appear in these earlier works.

\paragraph{Continuous-time perspective on high-order tensor algorithms.} To the best of our knowledge, all the existing work on continuous-time interpretations of high-order tensor algorithms focus on convex optimization~\citep{Wibisono-2016-Variational,Song-2021-Unified,Lin-2022-Control}. In particular,~\citet{Wibisono-2016-Variational} studied the following inertial gradient system with asymptotically vanishing damping:
\begin{equation*}
\ddot{x}(t) + \tfrac{p+2}{t}\dot{x}(t) + C(p+1)^2t^{p-1}\nabla\Phi(x(t)) = 0,
\end{equation*}
which is an open-loop system without Hessian-driven damping. They derived a class of $p^\textnormal{th}$-order tensor algorithms by implicit discretization and established a convergence rate of $O(k^{-(p+1)})$ in terms of the objective function gap.~\citet{Song-2021-Unified} proposed another form of open-loop dynamics (we consider the simplified form in a Euclidean setting):
\begin{equation*}
\ddot{x}(t) + \left(\tfrac{2\dot{a}(t)}{a(t)} - \tfrac{\ddot{a}(t)}{\dot{a}(t)}\right)\dot{x}(t) + \left(\tfrac{(\dot{a}(t))^2}{a(t)}\right)\nabla\Phi(x(t)) = 0, 
\end{equation*}
which is also open-loop and lacks Hessian-driven damping.  Recently,~\citet{Lin-2022-Control} provided a control-theoretic perspective on optimal acceleration for high-order tensor algorithms.  They considered the following closed-loop control system with Hessian-driven damping: 
\begin{equation*}
\ddot{x}(t) + \alpha(t)\dot{x}(t) + \beta(t)\nabla^2\Phi(x(t))\dot{x}(t) + b(t)\nabla\Phi(x(t)) = 0,
\end{equation*}
where $(\alpha, \beta, b)$ are defined by 
\begin{equation*}
\begin{array}{ll}
& \alpha(t) = \tfrac{2\dot{a}(t)}{a(t)} - \tfrac{\ddot{a}(t)}{\dot{a}(t)}, \quad \beta(t) = \tfrac{(\dot{a}(t))^2}{a(t)}, \quad b(t) = \tfrac{\dot{a}(t)(\dot{a}(t) + \ddot{a}(t))}{a(t)}, \\
& a(t) = \tfrac{1}{4}(\int_0^t \sqrt{\lambda(s)} ds)^2, \quad (\lambda(t))^p\|\nabla\Phi(x(t))\|^{p-1} = \theta,   
\end{array} 
\end{equation*}
and recovered a class of optimal high-order tensor algorithms~\citep{Monteiro-2013-Accelerated,Gasnikov-2019-Near} from implicit discretization of the above system.  Here the rate is $O(t^{-(3p+1)/2})$ in terms of the objective function gap. 

There is comparatively little work on the development of high-order tensor algorithms for monotone inclusion problems; indeed, we are only aware of the high-order mirror-prox method~\citep{Bullins-2022-Higher}.  However, the derivation of this algorithm does not flow from a single underlying principle but again involves case-specific algebra. It has been an open challenge to extend earlier work on open-loop and closed-loop systems from convex optimization to monotone inclusion problems.

\section{Convergence Properties of Trajectories}\label{sec:convergence}
In this section, we give a Lyapunov function for analyzing the convergence properties of solution trajectories of our system in Eq.~\eqref{sys:general} and Eq.~\eqref{sys:choice-feedback}.  We prove the weak convergence of trajectories to equilibrium by appeal to the Opial lemma as well as strong convergence results under additional conditions. We also derive new global convergence rates by estimating the rate of decrease of Lyapunov function.  Finally, we use another Lyapunov function to establish local linear convergence under an error bound condition. 

\subsection{Weak convergence}
We present our results on the weak convergence of trajectories. 
\begin{theorem}\label{Theorem:Weak-Convergence} 
Suppose that $(x, \lambda): [0, +\infty) \mapsto \HCal \times (0, +\infty)$ is a global solution of the closed-loop control system in Eq.~\eqref{sys:general} and Eq.~\eqref{sys:choice-feedback}. Then, there exists some $\bar{x} \in A^{-1}(0)$ such that the trajectory $x(t)$ weakly converges to $\bar{x}$ as $t \rightarrow +\infty$. 
\end{theorem}
\begin{remark}
In the Hilbert-space setting (possibly infinite dimensional),  the weak convergence of $x(\cdot)$ to some $\bar{x} \in A^{-1}(0)$ in Theorem~\ref{Theorem:Weak-Convergence} is the best we can expect without additional conditions.  The same result was established for other systems for convex optimization, including an open-loop inertial system for first-order algorithms~\citep{Attouch-2016-Fast} and a closed-loop control system (which is the special instance of our system for $p=2$), and for second-order algorithms~\citep{Attouch-2016-Dynamic}.  
\end{remark}

We define the following simple Lyapunov function for the system in Eq.~\eqref{sys:general} and Eq.~\eqref{sys:choice-feedback}: 
\begin{equation}\label{def:Lyapunov}
\ECal(t) = \tfrac{1}{2}\|x(t) - z\|^2,  
\end{equation}
where $z \in \HCal$ is a point in the Hilbert space. Note that this function measures the distance between $x(t)$ and any fixed point $z \in \HCal$. It is simpler than that used for analyzing the convergence of Newton-like inertial dynamics for monotone inclusion problems and different from the ones developed for the systems with asymptotically vanishing damping. The closest Lyapunov function to ours is the one employed by~\citet{Attouch-2016-Dynamic}, which is defined as the distance between $x(t)$ and $z \in A^{-1}(0)$.  We note that the seemingly minor modification to the form in Eq.~\eqref{def:Lyapunov} is key to deriving new results on the ergodic convergence of trajectories in terms of a gap function (see Theorem~\ref{Theorem:Trajectory-Convergence-Rate} and its proof). 

The following proposition gives the Opial lemma in its continuous form~\citep{Opial-1967-Weak}. It has become a basic analytical tool to study the weak convergence of trajectories of dynamical systems associated with discrete-time algorithms for convex optimization~\citep{Alvarez-2000-Minimizing, Attouch-2000-Heavy} and monotone inclusion~\citep{Attouch-2001-Second, Attouch-2016-Dynamic}. 
\begin{proposition}[Opial Lemma]\label{Prop:Opial}
Suppose that $S \subseteq \HCal$ is a nonempty subset and $x: [0, +\infty) \mapsto \HCal$ is a mapping.  Then, there exists some $\bar{x} \in S$ such that $x(t)$ weakly converges to $\bar{x}$ as $t \rightarrow +\infty$ if both of the following assumptions hold true:
\begin{enumerate}
\item For every $z \in S$, we have that $\lim_{t \rightarrow +\infty} \|x(t) - z\|$ exists. 
\item Every weak sequential cluster point of $x(\cdot)$ belongs to $S$. 
\end{enumerate}
\end{proposition}
To prove Theorem~\ref{Theorem:Weak-Convergence}, we provide one technical lemma that characterizes a descent property of $\ECal(\cdot)$ which is crucial to our subsequent analysis in this paper.  
\begin{lemma}\label{Lemma:Lyapunov-descent}
Suppose that $(x, \lambda): [0, +\infty) \mapsto \HCal \times (0, +\infty)$ is a global solution of the closed-loop control system in Eq.~\eqref{sys:general} and Eq.~\eqref{sys:choice-feedback} and let $z \in A^{-1}(0)$ in Eq.~\eqref{def:Lyapunov}. Then, we have
\begin{equation*}
\tfrac{d\ECal(t)}{dt} \leq - \|x(t) - (I + \lambda(t)A)^{-1}x(t)\|^2. 
\end{equation*}
\end{lemma}
\paragraph{Proof of Theorem~\ref{Theorem:Weak-Convergence}:} By Proposition~\ref{Prop:Opial}, it suffices to prove that (i) $\lim_{t \rightarrow +\infty} \|x(t) - z\|$ exists for every $z \in A^{-1}(0)$ and (ii) every weak sequential cluster point of $x(\cdot)$ belongs to $A^{-1}(0)$. 

By Lemma~\ref{Lemma:Lyapunov-descent}, we have $\ECal(t) = \tfrac{1}{2}\|x(t) - z\|^2$ is nonincreasing for any fixed $z \in A^{-1}(0)$.  This implies that (i) holds true.  Further,  let $\bar{x} \in \HCal$ be a weak sequential cluster point of $x(\cdot)$, we claim that it is also a weak sequential cluster point of $y(\cdot) = (I + \lambda(\cdot)A)^{-1}x(\cdot)$. Indeed, Lemma~\ref{Lemma:Lyapunov-descent} implies
\begin{equation*}
\ECal(0) - \ECal(t) \geq \int_0^t \|x(s) - y(s)\|^2 \; ds, \quad \textnormal{for all } t \geq 0. 
\end{equation*}
Since $\ECal(t) \geq 0$, we have $\ECal(0) - \ECal(t) \leq \ECal(0)$.  As a direct consequence of Lemma~\ref{Lemma:control-second} and Theorem~\ref{Theorem:Global-Existence-Uniquess}, we have that $\lambda: [0, +\infty)$ is nondecreasing.  This together with Eq.~\eqref{sys:choice-feedback} implies that $t \mapsto \|x(t) - y(t)\|$ is nonincreasing.  Putting these pieces together yields 
\begin{equation*}
\|x(t) - y(t)\|^2 \leq \tfrac{\ECal(0)}{t} \rightarrow 0, \quad \textnormal{as } t \rightarrow +\infty. 
\end{equation*}
This implies the desired result that $\bar{x}$ is also a weak sequential cluster point of $y(\cdot)$. In addition, we have $\frac{1}{\lambda(t)}(x(t) - y(t)) \in Ay(t)$.  Combining $\|x(t) - y(t)\| \rightarrow 0$ and Eq.~\eqref{sys:choice-feedback} implies that $\lambda(t) \rightarrow +\infty$. Therefore, we have that $\frac{1}{\lambda(t)}\|x(t) - y(t)\| \rightarrow 0$ as $t \rightarrow +\infty$. Since $\bar{x}$ is a weak sequential cluster point of $y(\cdot)$ and the graph of $A$ is demi-closed~\citep[Chapter 10]{Goebel-1990-Topics}, we have $0 \in A\bar{x}$ and hence $\bar{x} \in A^{-1}(0)$. 

\subsection{Strong convergence}
We further establish strong convergence of a global solution of the closed-loop control system in Eq.~\eqref{sys:general} and Eq.~\eqref{sys:choice-feedback} under additional conditions.  
\begin{theorem}\label{Theorem:Strong-Convergence} 
Suppose that $(x, \lambda): [0, +\infty) \mapsto \HCal \times (0, +\infty)$ is a global solution of the closed-loop control system in Eq.~\eqref{sys:general} and Eq.~\eqref{sys:choice-feedback}. Then, there exists some $\bar{x} \in A^{-1}(0)$ such that the trajectory $x(t)$  converges strongly to $\bar{x}$ as $t \rightarrow +\infty$ if either of the following conditions holds true:
\begin{enumerate}
\item $A = \nabla \Phi$ where $\Phi: \HCal \mapsto \br \cup \{+\infty\}$ is convex, differentiable and inf-compact.\footnote{A function $\Phi$ is \emph{inf-compact} if for any $r > 0$ and $\kappa \in \br$, the set $\{x \in \HCal: \|x\| \leq r, \Phi(x) \leq \kappa\}$ is a relatively compact set in $\HCal$, i.e., the set whose closure is compact.}  
\item $A^{-1}(0)$ has a nonempty interior. 
\end{enumerate}
\end{theorem}
\begin{remark}
In the Hilbert-space setting, the strong convergence is desirable since it guarantees that $\|x(t) - \bar{x}\|$ eventually becomes arbitrarily small~\citep{Bauschke-2001-Weak}. It has been studied for various discrete-time algorithms in convex optimization~\citep{Solodov-2000-Forcing} and the realization of the importance of strong convergence dates to~\citet{Guler-1991-Convergence} who showed that the convergence rate of the sequence of objectives $\{\Phi(x_k)\}_{k \geq 0}$ is better when $\{x_k\}_{k \geq 0}$ with strong convergence than weak convergence.  In addition, the conditions assumed in Theorem~\ref{Theorem:Strong-Convergence} can be verifiable by hand and the similar results can be obtained using the generalized chain rule for the subdifferential in the case $A = \partial \Phi$. 
\end{remark}
\paragraph{Proof of Theorem~\ref{Theorem:Strong-Convergence}:} For the first case, we claim that $t \mapsto \Phi(x(t))$ is nonincreasing.  Indeed, we let $y(t) = (I + \lambda(t)\nabla \Phi)^{-1}x(t)$ and deduce from the convexity of $\Phi$ that 
\begin{equation*}
\tfrac{d \Phi(x(t))}{dt} = \langle\dot{x}(t), \nabla \Phi(x(t))\rangle \overset{\textnormal{Eq.~\eqref{sys:general}}}{=} \langle y(t) - x(t), \nabla \Phi(x(t))\rangle \leq \langle y(t) - x(t), \nabla \Phi(y(t))\rangle. 
\end{equation*}
By the definition of $y(t)$, we have $\lambda(t)\nabla \Phi(y(t)) + y(t) - x(t) = 0$.  This implies
\begin{equation*}
\langle y(t) - x(t), \nabla \Phi(y(t))\rangle = -\lambda(t)\|\nabla \Phi(y(t))\|^2 \leq 0. 
\end{equation*}
Putting these pieces together yields the desired result. As such, it is immediate to see that $x(\cdot)$ is contained in $\{x \in \HCal: \Phi(x) \leq \Phi(x_0)\}$.  By Lemma~\ref{Lemma:Lyapunov-descent}, we have $\ECal(t) = \tfrac{1}{2}\|x(t) - z\|^2$ is nonincreasing for any fixed $z \in A^{-1}(0)$.  Thus, letting $x^\star \in A^{-1}(0)$ with $\|x^\star\|$ finite, we have that $\ECal(0)$ is finite and 
\begin{equation*}
x(t) \in S_0 \doteq \left\{x \in \HCal: \Phi(x) \leq \Phi(x_0), \|x\| \leq \|x^\star\| + \sqrt{2\ECal(0)}\right\}, \quad \textnormal{for all } t \in [0, +\infty). 
\end{equation*}
Since $\Phi: \HCal \mapsto \br \cup \{+\infty\}$ is inf-compact on any bounded set, we have that $S_0$ is relatively compact.  This implies that the trajectory $x(\cdot)$ is relatively compact.  By Theorem~\ref{Theorem:Weak-Convergence}, there exists some $\bar{x} \in A^{-1}(0)$ such that $x(t)$ converges weakly to $\bar{x}$ as $t \rightarrow +\infty$.  As such, we conclude the desired result. 

For the second case,  letting $x^\star \in \HCal$ be a point in the interior of $A^{-1}(0)$, there exists $\delta > 0$ such that $\BB_\delta(x^\star) \subseteq A^{-1}(0)$.  Denoting $A_\lambda = I - (I + \lambda A)^{-1}$, we have
\begin{equation*}
x \in A^{-1}(0) \Longleftrightarrow 0 \in Ax \Longleftrightarrow x = (I + \lambda A)^{-1}x \Longleftrightarrow 0 = A_\lambda x \Longleftrightarrow x \in A_\lambda^{-1}(0),
\end{equation*}
which implies that $A^{-1}(0) = A_\lambda^{-1}(0)$ and $\BB_\delta(x^\star) \subseteq A_\lambda^{-1}(0)$ for any $\lambda > 0$. It is also well known that $A_\lambda$ is monotone~\citep{Rockafellar-1970-Convex} Since $\BB_\delta(x^\star) \subseteq A_\lambda^{-1}(0)$, we have $x^\star + \delta h \in A_\lambda^{-1}(0)$ for any $h \in \HCal$ with $\|h\| \leq 1$. This together with the monotonicity of $A_\lambda$ yields
\begin{equation*}
\langle A_\lambda x(t), x(t) - (x^\star + \delta h)\rangle \geq 0, 
\end{equation*}
which implies
\begin{equation}\label{inequality:Strong-convergence-first}
\delta \langle A_\lambda x(t), h\rangle \leq \langle A_\lambda x(t), x(t) - x^\star\rangle. 
\end{equation}
Combining the above inequality with Eq.~\eqref{sys:general} yields
\begin{equation*}
\|\dot{x}(t)\| = \|A_{\lambda(t)}x(t)\| = \sup_{\|h\| \leq 1} \langle A_{\lambda(t)} x(t), h\rangle \overset{\textnormal{Eq.~\eqref{inequality:Strong-convergence-first}}}{\leq} \tfrac{1}{\delta}\langle A_{\lambda(t)} x(t), x(t) - x^\star\rangle = -\tfrac{1}{\delta}\langle \dot{x}(t), x(t) - x^\star\rangle. 
\end{equation*}
Then, we let $0 \leq t_1 \leq t_2$ and deduce from the above inequality that 
\begin{equation*}
\|x(t_2) - x(t_1)\| \leq \int_{t_1}^{t_2} \|\dot{x}(s)\| \; ds \leq - \tfrac{1}{\delta}\left(\int_{t_1}^{t_2} \langle \dot{x}(s), x(s) - x^\star\rangle \right) \leq \tfrac{1}{2\delta}\left(\|x(t_1) - x^\star\|^2 - \|x(t_2) - x^\star\|^2\right). 
\end{equation*}
Since $x^\star \in A^{-1}(0)$, we deduce from Lemma~\ref{Lemma:Lyapunov-descent} that $\|x(t) - x^\star\|$ is nonincreasing and convergent. Thus, the trajectory $x(\cdot)$ has the Cauchy property which implies the desired result. 

\subsection{Rate of convergence}
We prove the ergodic convergence rate of $O(t^{-(p+1)/2})$ for a global solution of the closed-loop control system in Eq.~\eqref{sys:general} and Eq.~\eqref{sys:choice-feedback} in terms of a gap function.  We also prove a pointwise convergence rate of $O(t^{-p/2})$ in terms of a residue function, and then establish local linear convergence for a global solution in terms of a distance function. 

Before stating our results, we provide the gap function and the residue function for monotone inclusion problems. Indeed, the following gap function originates from the Fitzpatrick function~\citep{Borwein-2010-Convex} and is also defined in the concurrent work of~\citet{Cui-2022-Stochastic}. Formally, we have
\begin{equation}\label{def:gap-function}
\textsc{gap}(x) = \sup_{z \in \dom(A)} \sup_{\xi \in Az} \ \langle \xi, x - z\rangle. 
\end{equation}
Clearly, $\textsc{gap}(\cdot)$ is closed\footnote{A function $\Phi: \HCal \mapsto \br$ is closed if the sublevel set $\{x \in \HCal: \Phi(x) \leq \alpha\}$ is closed for any $\alpha \in \br$; see~\citet{Rockafellar-1970-Convex}. } and convex.  Moreover, if $A$ is maximal monotone, we have that $\textsc{gap}(x) \geq 0$ for all $x \in \HCal$ with equality if and only if $x \in A^{-1}(0)$ holds. The residue function is derived from the monotone inclusion problem as follows, 
\begin{equation}\label{def:residue-function}
\textsc{res}(x) = \inf_{\xi \in Ax} \|\xi\|. 
\end{equation}
We are now ready to present our main results on the global convergence rate estimation in terms of the gap function in Eq.~\eqref{def:gap-function} and the residue function in Eq.~\eqref{def:residue-function}. 
\begin{theorem}\label{Theorem:Trajectory-Convergence-Rate}
Suppose that $(x, \lambda): [0, +\infty) \mapsto \HCal \times (0, +\infty)$ is a global solution of the closed-loop control system in Eq.~\eqref{sys:general} and Eq.~\eqref{sys:choice-feedback} and let $\dom(A)$ be closed and bounded. Then,  we have
\begin{equation*}
\textsc{gap}(\tilde{z}(t)) = O(t^{-\frac{p+1}{2}}),  
\end{equation*}
and 
\begin{equation*}
\textsc{res}(z(t)) = O(t^{-\frac{p}{2}}),  
\end{equation*}
where $\tilde{z}(\cdot)$ and $z(\cdot)$ are uniquely determined by $\lambda(\cdot)$ and $x(\cdot)$ as follows, 
\begin{align*}
\textnormal{\bf (Ergodic Iterate)} \qquad & \tilde{z}(t) = \tfrac{1}{\int_0^t \lambda(s) \; ds}\left(\int_0^t \lambda(s)(I + \lambda(s)A)^{-1}x(s) \; ds\right),\\
\textnormal{\bf (Pointwise Iterate)} \qquad & z(t) = (I + \lambda(t)A)^{-1}x(t). 
\end{align*}
\end{theorem}
\begin{remark}
Theorem~\ref{Theorem:Trajectory-Convergence-Rate} is new to the best of our knowledge and extends several classical results concerning discrete-time algorithms for monotone inclusion problems. Indeed, the discrete-time version of our results have been obtained by the extragradient method for $p=1$~\citep{Nemirovski-2004-Prox,Monteiro-2010-Complexity,Monteiro-2011-Complexity} and the Newton proximal extragradient method for $p=2$~\citep{Monteiro-2012-Iteration}. A similar ergodic convergence result was achieved by high-order mirror-prox method~\citep{Bullins-2022-Higher} for saddle point and VI problems for $p \geq 3$. Notably, our theorem demonstrates the importance of averaging for monotone inclusion problems by showing that the convergence rate can be faster in the ergodic sense for all $p \geq 1$.  The idea of averaging for convex optimization and monotone VIs goes back to at least the mid-seventies~\citep{Bruck-1977-Weak,Lions-1978-Methode,Nemirovski-1978-Cesari,Nemirovski-1981-Effective}. Its advantage was also recently justified for saddle point problems and VIs by establishing lower bounds~\citep{Golowich-2020-Tight, Ouyang-2021-Lower}. Our theorem provides another way to understand averaging from a continuous-time point of view. 
\end{remark}

We define the so-called error bound condition as an inequality that bounds the distance between $x \in \HCal$ and $A^{-1}(0)$ by a residual function at $x$. This condition has been proven to be useful in proving the linear convergence of discrete-time algorithms for solving convex optimization and monotone VI problems~\citep{Lewis-1998-Error,Drusvyatskiy-2018-Error, Drusvyatskiy-2021-Nonsmooth}. We adapt this condition for monotone inclusion problems as follows.  We assume that there exists $\delta > 0$ and $\kappa > 0$ such that  
\begin{equation}\label{def:EB}
\textsc{dist}(0, Ax) \leq \delta \Longrightarrow \textsc{dist}(x, A^{-1}(0)) \leq \kappa \cdot \textsc{dist}(0, Ax), 
\end{equation}
where $\textsc{dist}(x, S) = \inf_{z \in S} \|x - z\|$ is a distance function. The corresponding Lyapunov function used for analyzing a global solution under the error bound condition is as follows:
\begin{equation}\label{def:EB-Lyapunov}
\tilde{\ECal}(t) = \frac{1}{2} (\textsc{dist}(x(t), A^{-1}(0)))^2 \doteq \inf_{z \in A^{-1}(0)} \left\{\frac{1}{2}\|x(t) - z\|^2\right\}. 
\end{equation}
The Lyapunov function in Eq.~\eqref{def:EB-Lyapunov} can be interpreted as a continuous version of a function used by various authors; see e.g.,~\citet{Tseng-1995-Linear}. The convergence rate estimation intuitively depends on the descent inequality. This requires the differentiation of $\tilde{\ECal}(\cdot)$ which is not immediate since $A^{-1}(0)$ is not a singleton set and the projection of $x(t)$ onto $A^{-1}(0)$ will change as $t$ varies. We instead upper bound the difference $\tilde{\ECal}(t') - \ECal(t)$ for any $t' \geq t$ given a fixed $t$. We have the following theorem.
\begin{theorem}\label{Theorem:Trajectory-Convergence-Linear}
Suppose that $(x, \lambda): [0, +\infty) \mapsto \HCal \times (0, +\infty)$ is a global solution of the closed-loop control system in Eq.~\eqref{sys:general} and Eq.~\eqref{sys:choice-feedback} and let the error bound condition in Eq.~\eqref{def:EB} hold true.  Then, there exists a sufficiently large $t_0 > 0$ such that 
\begin{equation*}
\textsc{dist}(x(t), A^{-1}(0)) = O(e^{-ct/2}), \quad \textnormal{for all } t > t_0. 
\end{equation*}
where $c > 0$ is a constant and upper bounded by $c \leq 2\left(1 + \frac{\kappa}{\lambda(0)}\right)^{-2} $. 
\end{theorem}
\begin{remark}
Theorem~\ref{Theorem:Trajectory-Convergence-Linear} establishes the strong convergence of $x(\cdot)$ to some $x^\star \in A^{-1}(0)$ under the error bound condition and establishes local linear convergence in terms of a distance function.  This improves the results in Theorem~\ref{Theorem:Trajectory-Convergence-Rate} and demonstrates the value of the error bound condition. The same linear convergence guarantee is established in~\citet{Csetnek-2021-Convergence} under similar conditions. In fact, the convergence analysis of discrete-time algorithms under an error bound condition is of independent interest~\citep{Solodov-2003-Convergence} and its analysis involves different techniques.
\end{remark}
\paragraph{Proof of Theorem~\ref{Theorem:Trajectory-Convergence-Rate}:} Using the definition of $\ECal(\cdot)$ in Eq.~\eqref{def:Lyapunov} and the same argument as applied in Lemma~\ref{Lemma:Lyapunov-descent}, we have
\begin{equation*}
\tfrac{d\ECal(t)}{dt} = -\|x(t) - (I + \lambda(t)A)^{-1}x(t)\|^2 - \langle x(t) - (I + \lambda(t)A)^{-1}x(t), (I + \lambda(t)A)^{-1}x(t) - z\rangle.  
\end{equation*}
Using the definition of $z(\cdot)$ and the fact that $\|x(t) - (I + \lambda(t)A)^{-1}x(t)\|^2 \geq 0$, we have
\begin{equation*}
\tfrac{d\ECal(t)}{dt} \leq - \langle x(t) - z(t), z(t) - z\rangle. 
\end{equation*}
Since $A$ is monotone and $\frac{1}{\lambda(t)}(x(t) - z(t)) \in Az(t)$, we have
\begin{equation*}
\langle x(t) - z(t), z(t) - z\rangle \geq \lambda(t)\langle \xi, z(t) - z\rangle, \quad \textnormal{for all } \xi \in Az. 
\end{equation*}
Putting these pieces together yields that, for any $z \in \dom(A)$ and any $\xi \in Az$, we have
\begin{equation*}
\tfrac{d\ECal(t)}{dt} \leq - \lambda(t)\langle \xi, z(t) - z\rangle. 
\end{equation*}
Integrating this inequality over $[0, t]$ yields
\begin{equation*}
\int_0^t \lambda(s)\langle \xi, z(s) - z\rangle \; ds \leq \ECal(0) - \ECal(t) \leq \tfrac{1}{2}\|x_0 - z\|^2, \quad \textnormal{for all } t \geq 0.  
\end{equation*}
Equivalently,  we have
\begin{equation*}
\langle\xi, \tilde{z}(t) - z\rangle \leq \tfrac{1}{\int_0^t \lambda(s) \; ds}\left(\tfrac{1}{2}\|x_0 - z\|^2\right) \leq \tfrac{1}{\int_0^t \lambda(s) \; ds}\left(\sup_{z \in \dom(A)} \|x_0 - z\|^2\right). 
\end{equation*}
By the definition of $\textsc{gap}(\cdot)$ and using the boundedness of $\dom(A)$, we have
\begin{equation}\label{inequality:trajectory-rate-first}
\textsc{gap}(\tilde{z}(t)) = O\left(\tfrac{1}{\int_0^t \lambda(s) \; ds}\right). 
\end{equation}
Since $z(t) = (I + \lambda(t)A)^{-1}x(t)$, we can obtain from the proof of Theorem~\ref{Theorem:Weak-Convergence} that 
\begin{equation}\label{inequality:trajectory-rate-second}
\|x(t) - z(t)\|^2 \leq \tfrac{\ECal(0)}{t}, \quad \textnormal{for all } t \geq 0.  
\end{equation}
Since $\frac{1}{\lambda(t)}(x(t) - z(t)) \in Az(t)$, we have
\begin{equation}\label{inequality:trajectory-rate-third}
\textsc{res}(z(t)) \leq \tfrac{1}{\lambda(t)}\|x(t) - z(t)\| \leq \tfrac{\ECal(0)}{\lambda(t)\sqrt{t}} = O\left(\tfrac{1}{\lambda(t)\sqrt{t}}\right).
\end{equation}
It remain to estimate the lower bound for the feedback law $\lambda(\cdot)$.  Indeed, by combining Eq.~\eqref{inequality:trajectory-rate-second} and the algebraic equation in Eq.~\eqref{sys:choice-feedback}, we have
\begin{equation}\label{inequality:trajectory-rate-fourth}
\lambda(t) = \tfrac{\theta}{\|x(t) - z(t)\|^{p-1}} \geq \theta\left(\tfrac{t}{\ECal(0)}\right)^{\frac{p-1}{2}}. 
\end{equation}
Plugging Eq.~\eqref{inequality:trajectory-rate-fourth} into Eq.~\eqref{inequality:trajectory-rate-first} and Eq.~\eqref{inequality:trajectory-rate-third} yields the desired results. 

\paragraph{Proof of Theorem~\ref{Theorem:Trajectory-Convergence-Linear}:} Fixing $t \geq 0$ and using the definition of $\tilde{\ECal}$ in Eq.~\eqref{def:EB-Lyapunov}, we have
\begin{equation*}
\tilde{\ECal}(t') - \tilde{\ECal}(t) = \tfrac{1}{2}\left((\textsc{dist}(x(t'), A^{-1}(0)))^2 - (\textsc{dist}(x(t), A^{-1}(0)))^2\right), \quad \textnormal{for all } t' \geq t. 
\end{equation*}
We let $x^\star(t)$ denote the projection of $x(t)$ onto $A^{-1}(0)$ and deduce from the above inequality that 
\begin{equation*}
\tfrac{\tilde{\ECal}(t') - \tilde{\ECal}(t)}{t' - t} \leq \tfrac{\|x(t') - x^\star(t)\|^2 - \|x(t) - x^\star(t)\|^2}{2(t' - t)} = \langle \tfrac{x(t') - x(t)}{t' - t},  \tfrac{x(t') + x(t)}{2} - x^\star(t)\rangle. 
\end{equation*}
Letting $t' \rightarrow^+ t$, we have
\begin{equation}\label{inequality:trajectory-linear-first}
\limsup_{t' \rightarrow^+ t} \tfrac{\tilde{\ECal}(t') - \tilde{\ECal}(t)}{t' - t} \leq \langle \dot{x}(t),  x(t) - x^\star(t)\rangle. 
\end{equation}
For simplicity, we let $y(t) = (I + \lambda(t)A)^{-1}x(t)$ and deduce that $\frac{1}{\lambda(t)}(x(t) - y(t)) \in Ay(t)$.  In addition, $0 \in Ax^\star(t)$.  Putting these pieces together with the monotonicity of $A$ yields
\begin{eqnarray}\label{inequality:trajectory-linear-second}
\lefteqn{\langle \dot{x}(t),  x(t) - x^\star(t)\rangle \overset{\textnormal{Eq.~\eqref{sys:general}}}{=} - \langle x(t) - y(t), x(t) - x^\star(t)\rangle} \\
& = & - \|x(t) - y(t)\|^2 - \langle x(t) - y(t), y(t) - x^\star(t)\rangle \leq - \|x(t) - y(t)\|^2. \nonumber
\end{eqnarray}
Using the same argument in the proof of Theorem~\ref{Theorem:Weak-Convergence}, we have $t \mapsto \frac{1}{\lambda(t)}\|x(t) - y(t)\|$ is nonincreasing and converges to zero as $t \rightarrow +\infty$.  So there exists a sufficiently large $t_0 > 0$ such that 
\begin{equation*}
\tfrac{1}{\lambda(t)}\|x(t) - y(t)\| \leq \delta,  \quad \textnormal{for all } t \geq t_0,
\end{equation*}
where $\delta > 0$ is defined in the error bound condition (cf. Eq.~\eqref{def:EB}). Recall that $\frac{1}{\lambda(t)}(x(t) - y(t)) \in Ay(t)$, we have $\textsc{dist}(0, Ay(t)) \leq \delta$. Since the error bound condition in Eq.~\eqref{def:EB} holds true, we have
\begin{equation*}
\textsc{dist}(y(t), A^{-1}(0)) \leq \kappa \cdot \textsc{dist}(0, Ay(t)), 
\end{equation*}
We let $y^\star(t)$ denote the projection of $y(t)$ onto $A^{-1}(0)$ and deduce from the triangle inequality that 
\begin{equation*}
\textsc{dist}(x(t), A^{-1}(0)) \leq \|x(t) - y^\star(t)\| \leq \|x(t) - y(t)\| + \textsc{dist}(y(t), A^{-1}(0)). 
\end{equation*}
Putting these pieces together yields
\begin{equation*}
\textsc{dist}(x(t), A^{-1}(0)) \leq \|x(t) - y(t)\| + \kappa \cdot \textsc{dist}(0, Ay(t)) \leq \left(1 + \tfrac{\kappa}{\lambda(t)}\right)\|x(t) - y(t)\|. 
\end{equation*} 
Since $\lambda(t)$ is nondecreasing (cf. Lemma~\ref{Lemma:control-second}),  we have $\lambda(t) \geq \lambda(0)$.  By the definition of $\tilde{\ECal}$, we have
\begin{equation}\label{inequality:trajectory-linear-third}
\tilde{\ECal}(t) = \tfrac{1}{2}(\textsc{dist}(x(t), A^{-1}(0)))^2 \leq \tfrac{1}{2}\left(1 + \tfrac{\kappa}{\lambda(0)}\right)^2\|x(t) - y(t)\|^2. 
\end{equation}
Plugging Eq.~\eqref{inequality:trajectory-linear-second} and Eq.~\eqref{inequality:trajectory-linear-third} into Eq.~\eqref{inequality:trajectory-linear-first}, we have
\begin{equation}\label{inequality:trajectory-linear-fourth}
\limsup_{t' \rightarrow^+ t} \tfrac{\tilde{\ECal}(t') - \tilde{\ECal}(t)}{t' - t} \leq -2\left(1 + \tfrac{\kappa}{\lambda(0)}\right)^{-2}\tilde{\ECal}(t) \leq -c \cdot \tilde{\ECal}(t). 
\end{equation}
Fixing $t > 0$, we define a partition of an interval $[0, t)$, 
\begin{equation*}
0 =t_0 < t_1 < t_2 < \ldots < t_i < \ldots < t_n = t, 
\end{equation*}
with $\sup_{0 \leq i \leq n-1} |t_{i+1} - t_i| \leq h$. Here, $h > 0$ is sufficiently small such that Eq.~\eqref{inequality:trajectory-linear-fourth} guarantees
\begin{equation*}
\tfrac{\tilde{\ECal}(t_{i+1}) - \tilde{\ECal}(t_i)}{t_{i+1} - t_i} \leq -c \cdot \tilde{\ECal}(t_i), \quad \textnormal{for all } i \in \{0, 1, 2, \ldots, n-1\}. 
\end{equation*}
This inequality implies 
\begin{equation*}
\tilde{\ECal}(t) - \tilde{\ECal}(0) = \sum_{i=0}^{n-1} (\tilde{\ECal}(t_{i+1}) - \tilde{\ECal}(t_i)) \leq -c \cdot \left(\sum_{i=0}^{n-1} \tilde{\ECal}(t_i) (t_{i+1} - t_i) \right). 
\end{equation*}
Since $\tilde{\ECal}(\cdot): [0, +\infty) \rightarrow [0, +\infty)$ is a continuous function, it is integrable (possibly not differentiable). Letting $h \rightarrow 0$, we have 
\begin{equation*}
\sum_{i=0}^{n-1} \tilde{\ECal}(t_i) (t_{i+1} - t_i)  \rightarrow \int_0^t \tilde{\ECal}(s) \; ds. 
\end{equation*}
Putting these pieces together yields
\begin{equation*}
\tilde{\ECal}(t) - \tilde{\ECal}(0) \leq -c\left(\int_0^t \tilde{\ECal}(s) \; ds\right). 
\end{equation*}
Recall the Gr\"{o}nwall–Bellman inequality in the integral form~\citep{Gronwall-1919-Note,Bellman-1943-Stability}: if $u(\cdot)$ and $\beta(\cdot)$ are both continuous and satisfy the integral inequality: $u(t) \leq u_0 + \int_0^t \beta(s) u(s) ds$, 
we have 
\begin{equation*}
u(t)\leq u_0 \exp\left(\int_0^t \beta(s) \; ds\right).
\end{equation*}
This implies that $\tilde{\ECal}(t) \leq \tilde{\ECal}(0)e^{-ct}$.  Therefore, we conclude that there exists a sufficiently large $t_0 > 0$ such that 
\begin{equation*}
\textsc{dist}(x(t), A^{-1}(0)) = O(e^{-ct/2}), \quad \textnormal{for all } t > t_0. 
\end{equation*}
This completes the proof. 

\subsection{Discussion}\label{subsec:Lyapunov-dis}
We comment on the main techniques for analyzing the system in Eq.~\eqref{sys:general} and Eq.~\eqref{sys:choice-feedback}, including Lyapunov analysis and weak versus strong convergence.  We also compare our approach to other approaches based on time scaling and dry friction. 

\paragraph{Lyapunov analysis.} Key to the continuous-time approach is to derive inertial gradient systems as limits of discrete-time algorithms and interpret the acceleration as the effect of asymptotically vanishing damping and Hessian-driven damping. Analyzing such a dynamical system requires a more complicated Lyapunov function than Eq.~\eqref{def:Lyapunov}. In this context,~\citet{Wilson-2021-Lyapunov} have constructed a unified Lyapunov function and their analysis was shown to be equivalent to Nesterov’s estimate sequence analysis for a variety of first-order algorithms, including quasi-monotone subgradient, accelerated gradient descent and conditional gradient. In contrast, the associated dynamical systems for general monotone inclusion problems need not contain any inertial term~\citep{Attouch-2011-Continuous,Attouch-2013-Global,Abbas-2014-Newton,Attouch-2016-Dynamic}. However, this does not mean that inertia is not relevant outside optimization. Indeed, the inertial dynamical systems and their discretization~\citep{Attouch-2011-Asymptotic} give a family of accelerated first-order algorithms for monotone inclusion problems under the cocoercive condition. Nonetheless, the Lyapunov analysis in the current paper becomes quite simple since our system does not involve any inertial term. In addition, the analysis of the convergence rate estimation under an error bound condition involves a new Lyapunov function that can be of independent interest.  

\paragraph{Weak versus strong convergence.} In the Hilbert-space setting, the (generalized) steepest descent dynamical system associated to a convex potential function $\Phi$ has the following form: 
\begin{equation*}
\left\{\begin{array}{ll}
& -\dot{x}(t) \in \partial \Phi(x(t)), \\
& x(0) = x_0. 
\end{array}\right.  
\end{equation*}
It is well known that the trajectory converges to a point $\bar{x} \in \{x: f(x) = \inf_{x \in \HCal} f(x)\} \neq \emptyset$~\citep{Brezis-1973-Operateurs,Brezis-1978-Asymptotic}. However, the theoretical understanding is far from being complete. In particular, it remains open how to characterize the relationship between $\bar{x}$ and the initial point $x_0$~\citep{Lemaire-1996-Asymptotical}. There is also a famous counterexample~\citep{Baillon-1978-Exemple} which shows that the trajectories of the above system converge weakly but not strongly. Despite the progress on weak versus strong convergence of a regularized Newton dynamic for monotone inclusion problems~\citep{Attouch-2018-Weak}, we are not aware of any discussion about these properties for closed-loop control systems and consider it an interesting open problem to find a counterexample (weak versus strong convergence) for the system in Eq.~\eqref{sys:general} and Eq.~\eqref{sys:choice-feedback}. 

It is worth mentioning that the convergence results of trajectories are important aspects of the
convergence analysis, especially in an infinite-dimensional setting; indeed,  these results have been established for the trajectories of various dynamical systems for monotone inclusion problems in earlier research~\citep{Attouch-2011-Continuous,Attouch-2013-Global,Attouch-2016-Dynamic,Abbas-2014-Newton,Bot-2016-Second,Attouch-2019-Convergence,Attouch-2020-Convergence,Attouch-2020-Newton,Attouch-2021-Continuous}.  A few results are valid only for weak convergence and become true for strong convergence only under additional conditions.  Some results are only valid in the ergodic sense, e.g., the rate of $O(t^{-(p+1)/2})$ in Theorem~\ref{Theorem:Trajectory-Convergence-Rate}.

\paragraph{Time scaling and dry friction.} In the context of dissipative dynamical systems associated with convex optimization algorithms,  there have been two simple yet universally powerful techniques to strengthen the convergence properties of trajectories: time scaling~\citep{Attouch-2019-Time,Attouch-2022-Fast} and dry friction~\citep{Adly-2020-Finite,Adly-2022-First}. In particular, the effect of time scaling is revealed by the coefficient parameter $b(t)$ which comes in as a factor of $\nabla\Phi(x(t))$ in the following open-loop inertial gradient system:
\begin{equation*}
\ddot{x}(t) + \alpha(t)\dot{x}(t) + \beta(t)\nabla^2\Phi(x(t))\dot{x}(t) + b(t)\nabla\Phi(x(t)) = 0. 
\end{equation*}  
In~\citet{Attouch-2019-Time}, the authors investigated the above system without Hessian-driven damping ($\beta(t)=0$). They proved that the convergence rate of a solution trajectory is $O(1/(t^2 b(t)))$ if $\alpha(\cdot)$ and $b(\cdot)$ satisfy certain conditions. As such, a clear improvement is attained by taking $b(t) \rightarrow +\infty$. This demonstrates the power and potential of time scaling, as further evidenced by recent work on systems with Hessian damping~\citep{Attouch-2022-Fast}. Furthermore, some recent work studied another open-loop inertial gradient system in the form of
\begin{equation*}
\ddot{x}(t) + \alpha(t)\dot{x}(t) + \partial \phi(\dot{x}(t)) + \beta(t)\nabla^2\Phi(x(t))\dot{x}(t) + b(t)\nabla\Phi(x(t)) \ni 0,
\end{equation*}  
where the dry friction function $\phi$ is convex with a sharp minimum at the origin,  e.g., $\phi(x) = r\|x\|$ with $r>0$.  In~\citet{Adly-2022-First},  the authors provided a study of the convergence of this system without Hessian-driven damping ($\beta(t)=0$) and derived a class of appealing first-order algorithms that achieve a finite convergence guarantee.  Subsequently,~\citet{Adly-2020-Finite} derived similar results for systems with Hessian-driven damping.

Unfortunately, the aforementioned works on time scaling and dry friction techniques are restricted to the study of open-loop systems associated with convex optimization algorithms.  As such, it remains unknown if these methodologies can be extended to monotone inclusion problems and further capture the continuous-time interpretation of acceleration in high-order monotone inclusion~\citep{Monteiro-2012-Iteration,Bullins-2022-Higher}. In contrast, our closed-loop control system provides a rigorous justification for the large-step condition in the algorithm of~\citet{Monteiro-2012-Iteration} and~\citet{Bullins-2022-Higher} when $p \geq 2$, explaining why the closed-loop control is key to acceleration in monotone inclusion. 

\section{Implicit Discretization and Acceleration}\label{sec:algorithm}
In this section, we propose an algorithmic framework that arises via implicit discretization of our system in Eq.~\eqref{sys:general} and Eq.~\eqref{sys:choice-feedback} in a Euclidean setting. It demonstrates the importance of the large-step condition~\citep{Monteiro-2012-Iteration} for acceleration in monotone inclusion problems, interpreting it as discretization of the algebraic equation. Our framework also clarifies why this condition is unnecessary for acceleration in monotone inclusion problems when $p=1$ (the algebraic equation vanishes). With an approximate tensor subroutine for smooth operator $A$, we derive a specific class of $p^\textnormal{th}$-order tensor algorithms which generalize $p^\textnormal{th}$-order tensor algorithms for convex-concave saddle point and monotone variational inequality problems~\citep{Bullins-2022-Higher}. 

\subsection{Conceptual algorithmic frameworks}\label{subsec:framework}
We study a conceptual algorithmic framework which is derived by implicit discretization of the closed-loop control system in Eq.~\eqref{sys:general} and Eq.~\eqref{sys:choice-feedback} in a Euclidean setting. Indeed, our system takes the form of 
\begin{equation*}
\left\{\begin{array}{ll}
& \dot{x}(t) + x(t) - (I + \lambda(t)A)^{-1}x(t) = 0, \\
& \lambda(t)\|(I + \lambda(t)A)^{-1}x(t) - x(t)\|^{p-1} = \theta, \\ 
& x(0) = x_0 \in \Omega. 
\end{array}\right.  
\end{equation*}
We define the discrete-time sequence $\{(x_k, \lambda_k)\}_{k \geq 0}$ that corresponds to its continuous-time counterpart $\{(x(t), \lambda(t))\}_{t \geq 0}$. By an implicit discretization, we have
\begin{equation}\label{sys:discrete}
\left\{\begin{array}{ll}
& x_{k+1} - (I + \lambda_{k+1}A)^{-1}x_k = 0, \\ 
& \lambda_{k+1}\|x_{k+1} - x_k\|^{p-1} = \theta, \\ 
& x_0 \in \Omega.   
\end{array}\right.
\end{equation}
By introducing two new variables $y_{k+1}$ and $v_{k+1} \in Ay_{k+1}$, the first and second lines of Eq.~\eqref{sys:discrete} can be equivalently reformulated as follows: 
\begin{equation*}
\left\{\begin{array}{ll}
& \lambda_{k+1}v_{k+1} + y_{k+1} - x_k = 0, \\ 
& \lambda_{k+1}\|y_{k+1} - x_k\|^{p-1} = \theta, \\ 
& x_{k+1} = x_k - \lambda_{k+1}v_{k+1}, \\
& x_0 \in \Omega.   
\end{array}\right.
\end{equation*}
We propose to solve the above equations inexactly with an accurate approximation of $A$. Following a suggestion of~\citet{Monteiro-2010-Complexity}, we introduce the relative error tolerance condition with an $\varepsilon$-enlargement of maximal monotone operators given by 
\begin{equation}\label{def:enlargement}
A^\epsilon(x) = \{v \in \br^d \mid \langle x - \tilde{x}, v - \tilde{v}\rangle \geq - \epsilon, \ \forall \tilde{x} \in \br^d, \forall \tilde{v} \in A\tilde{x}\}. 
\end{equation}
Our subroutine is to find $\lambda_{k+1} > 0$ and a triple $(y_{k+1}, v_{k+1}, \varepsilon_{k+1})$ such that 
\begin{equation*}
\|\lambda_{k+1}v_{k+1} + y_{k+1} - x_k\|^2 + 2\lambda_{k+1}\epsilon_{k+1} \leq  \sigma^2\|y_{k+1} - x_k\|^2, \quad v_{k+1} \in A^{\epsilon_{k+1}}(y_{k+1}). 
\end{equation*}
From the above condition, we see that $v_{k+1}$ is sufficiently close to an element in $A(y_{k+1})$. In addition, we relax the discrete-time algebraic equation by using $\lambda_{k+1}\|y_{k+1} - x_k\|^{p-1} \geq \theta$.  

We present our conceptual algorithmic framework formally in Algorithm~\ref{Algorithm:CAF}. It includes the large-step HPE framework of~\citet{Monteiro-2012-Iteration} as a special instance. In fact, we can recover the large-step HPE framework if we set $p = 2$ and change the notation of $A$ to $T$ in Algorithm~\ref{Algorithm:CAF}. 
\begin{algorithm}[!t]
\begin{algorithmic}\caption{Conceptual Algorithmic Framework}\label{Algorithm:CAF}
\STATE \textbf{STEP 0:}  Let $x_0, v_0 \in \br^d$, $\sigma \in (0, 1)$ and $\theta > 0$ be given, and set $k = 0$. 
\STATE \textbf{STEP 1:} If $0 \in Ax_k$, then \textbf{stop}. 
\STATE \textbf{STEP 2:} Otherwise, compute $\lambda_{k+1} > 0$ and a triple $(y_{k+1}, v_{k+1}, \epsilon_{k+1}) \in \HCal \times \HCal \times (0, +\infty)$ such that
\begin{align*}
& v_{k+1} \in A^{\epsilon_{k+1}}(y_{k+1}), \\
& \|\lambda_{k+1}v_{k+1} + y_{k+1} - x_k\|^2 + 2\lambda_{k+1}\epsilon_{k+1} \leq \sigma^2\|y_{k+1} - x_k\|^2, \\
& \lambda_{k+1}\|y_{k+1} - x_k\|^{p-1} \geq \theta.
\end{align*}
\STATE \textbf{STEP 3:} Compute $x_{k+1} = x_k - \lambda_{k+1}v_{k+1}$. 
\STATE \textbf{STEP 4:}  Set $k \leftarrow k+1$, and go to \textbf{STEP 1}. 
\end{algorithmic}
\end{algorithm}

\subsection{Convergence rate estimation}
We present both an ergodic and a pointwise estimate of convergence rate for Algorithm~\ref{Algorithm:CAF}. Our analysis is motivated by the aforementioned continuous-time analysis, simplifying the analysis in~\citet{Monteiro-2012-Iteration} for the case of $p=2$ and generalizing it to the case of $p>2$ in a systematic manner. 

We start with the presentation of our main results for Algorithm~\ref{Algorithm:CAF}, which generalizes the results in~\citet[Theorem~2.5 and~2.7]{Monteiro-2012-Iteration} in terms of a gap function for bounded domain from $p = 2$ to $p \geq 2$. To streamline the presentation, we rewrite a gap function in Eq.~\eqref{def:gap-function}: 
\begin{equation*}
\textsc{gap}(x) = \sup_{z \in \dom(A)} \sup_{\xi \in Az} \ \langle \xi, x - z\rangle. 
\end{equation*}
It is worth mentioning that the theoretical results in~\citet[Theorem~2.5 and~2.7]{Monteiro-2012-Iteration} are presented for unbounded domains using the modified optimality criterion. Our analysis can be extended using the relationship between a gap function and a relative tolerance error criterion~\citep{Monteiro-2010-Complexity}. However, the proof becomes significantly longer and its link with continuous-time analysis becomes unclear (the continuous-time version of the modified optimality criterion is unclear). Accordingly, we focus on the bounded domain and present the results for simplicity. 
\begin{theorem}\label{Theorem:CAF-main}
Let $k \geq 1$ be an integer and let $\dom(A)$ be closed and bounded. Then,  we have
\begin{equation*}
\textsc{gap}(\tilde{y}_k) = O(k^{-\frac{p+1}{2}}), 
\end{equation*}
and 
\begin{equation*}
\inf_{1 \leq i \leq k} \|v_i\| = O(k^{-\frac{p}{2}}),  \qquad \inf_{1 \leq i \leq k} \epsilon_i = O(k^{-\frac{p+1}{2}}), 
\end{equation*}
where the ergodic iterates $\{\tilde{y}_k\}_{k \geq 1}$ are defined by 
\begin{equation*}
\tilde{y}_k = \frac{1}{\sum_{i=1}^k \lambda_k}\left(\sum_{i=1}^k \lambda_i y_i\right). 
\end{equation*}
In addition, if we let $\epsilon_k = 0$ for all $k \geq 1$ and assume that the error bound condition in Eq.~\eqref{def:EB} holds true, the iterates $\{x_k\}_{k \geq 1}$ converge to $A^{-1}(0)$ with a local linear rate. 
\end{theorem}
Since the only difference between Algorithm~\ref{Algorithm:CAF} and the large-step HPE framework in~\citet{Monteiro-2012-Iteration} is the order in the algebraic equation, many technical results still hold for Algorithm~\ref{Algorithm:CAF} but their proofs tend to involve case-specific algebra.  Our contribution is to provide a simple proof which flows from the unified underlying continuous-time principle in Section~\ref{sec:convergence}, and also to derive local linear convergence under the error bound condition. 

We present a discrete-time Lypanunov function for Algorithm~\ref{Algorithm:CAF} as follows: 
\begin{equation}\label{def:Lyapunov-discrete}
\ECal_k = \tfrac{1}{2}\|x_k - z\|^2, 
\end{equation}
which will be used to prove technical results that pertain to Algorithm~\ref{Algorithm:CAF}.
\begin{lemma}\label{Lemma:CAF-descent}
For every integer $k \geq 1$, we have
\begin{equation}\label{inequality:CAF-descent-main}
\sum_{i=1}^k \lambda_i \langle v,  y_i - z\rangle + \tfrac{1 - \sigma^2}{2}\left(\sum_{i=1}^k \|x_{i-1} - y_i\|^2\right) \leq \ECal_0 - \ECal_k, \quad \textnormal{for all } v \in Az, 
\end{equation}
Letting $\tilde{y}_k = \frac{1}{\sum_{i=1}^k \lambda_i}(\sum_{i=1}^k \lambda_i y_i)$ be the ergodic iterates, we have $\sup_{v \in Az}\langle v, \tilde{y}_k - z\rangle \leq \frac{\ECal_0}{\sum_{i=1}^k \lambda_i}$. If we further assume that $\sigma < 1$, we have $\sum_{i=1}^k \|x_{i-1} - y_i\|^2 \leq \frac{\|x_0 - z^\star\|^2}{1 - \sigma^2}$ for any $z^\star \in A^{-1}(0)$. 
\end{lemma}
\begin{lemma}\label{Lemma:CAF-error}
For every integer $k \geq 1$ and $\sigma < 1$, there exists $1 \leq i \leq k$ such that
\begin{align*}
\inf_{1 \leq i \leq k} \sqrt{\lambda_i} \|v_i\| & \leq \sqrt{\tfrac{1+\sigma}{1-\sigma}}\left(\sum_{i=1}^k \lambda_i\right)^{-\frac{1}{2}}\left(\inf_{z^\star \in A^{-1}(0)}\|x_0 - z^\star\|\right),  \\ 
\inf_{1 \leq i \leq k} \epsilon_i & \leq \tfrac{\sigma^2}{2(1-\sigma^2)}\left(\sum_{i=1}^k \lambda_i\right)^{-1}\left(\inf_{z^\star \in A^{-1}(0)}\|x_0 - z^\star\|^2\right).  
\end{align*}
\end{lemma}
We provide a lemma giving a lower bound for $\sum_{i=1}^k \lambda_i$. The analysis is motivated by continuous-time analysis for the system in Eq.~\eqref{sys:general} and Eq.~\eqref{sys:choice-feedback}. 
\begin{lemma}\label{Lemma:CAF-control}
For $p \geq 1$ and every integer $k \geq 1$, we have
\begin{equation*}
\sum_{i=1}^k \lambda_i  \geq \theta\left((1 - \sigma^2)\left(\inf_{z^\star \in A^{-1}(0)}\|x_0 - z^\star\|^2\right)\right)^{\frac{p-1}{2}} k^{\frac{p+1}{2}}. 
\end{equation*}
\end{lemma}
\paragraph{Proof of Theorem~\ref{Theorem:CAF-main}:} For every integer $k \geq 1$, combining Lemma~\ref{Lemma:CAF-descent} and Lemma~\ref{Lemma:CAF-control} implies
\begin{equation*}
\textsc{gap}(\tilde{y}_k) = \sup_{z \in \dom(A)}\sup_{v \in Az}\langle v, \tilde{y}_k - z\rangle \leq \tfrac{1}{2(\sum_{i=1}^k \lambda_i)}\left(\sup_{z \in \dom(A)} \|z - x_0\|^2\right) = O(k^{-\frac{p+1}{2}}). 
\end{equation*}
Combining Lemma~\ref{Lemma:CAF-error} and Lemma~\ref{Lemma:CAF-control}, we have
\begin{align*}
\inf_{1 \leq i \leq k} \sqrt{\lambda_i} \|v_i\| & \leq \sqrt{\tfrac{1+\sigma}{1-\sigma}}\left(\sum_{i=1}^k \lambda_i\right)^{-\frac{1}{2}}\left(\inf_{z^\star \in A^{-1}(0)}\|x_0 - z^\star\|\right) = O(k^{-\frac{p+1}{4}}),  \\ 
\inf_{1 \leq i \leq k} \epsilon_i & \leq \tfrac{\sigma^2}{2(1-\sigma^2)}\left(\sum_{i=1}^k \lambda_i\right)^{-1}\left(\inf_{z^\star \in A^{-1}(0)}\|x_0 - z^\star\|^2\right) = O(k^{-\frac{p+1}{2}}).  
\end{align*}
From Step 2 of Algorithm 1, we have 
\begin{equation*}
\|\lambda_i v_i + y_i - x_{i-1}\|^2 + 2\lambda_i\epsilon_i \leq \sigma^2\|y_i - x_{i-1}\|^2, \quad \lambda_i\|y_i - x_{i-1}\|^{p-1} \geq \theta. 
\end{equation*}
Since $\lambda_i \geq 0$ and $\epsilon_i \geq 0$, the first inequality implies
\begin{equation*}
\|\lambda_i v_i + y_i - x_{i-1}\| \leq \sigma\|y_i - x_{i-1}\|. 
\end{equation*}
By the triangle inequality, we have 
\begin{equation*}
\sigma\|y_i - x_{i-1}\| \geq \|y_i - x_{i-1}\| - \lambda_i\|v_i\| \Longrightarrow \lambda_i\|v_i\| \geq (1-\sigma)\|y_i - x_{i-1}\|. 
\end{equation*}
This inequality together with $\lambda_i\|y_i - x_{i-1}\|^{p-1} \geq \theta$ implies
\begin{equation*}
\lambda_i\|v_i\|^{\frac{p-1}{p}} = (\lambda_i)^{\frac{1}{p}}(\lambda_i\|v_i\|)^{\frac{p-1}{p}} \geq \left(\tfrac{\theta}{\|y_i - x_{i-1}\|^{p-1}}\right)^{\frac{1}{p}}\left((1-\sigma)\|y_i - x_{i-1}\|\right)^{\frac{p-1}{p}} = \theta^{\frac{1}{p}}(1-\sigma)^{\frac{p-1}{p}}. 
\end{equation*}
Equivalently, we have 
\begin{equation*}
\sqrt{\lambda_i} \geq \theta^{\frac{1}{2p}}(1-\sigma)^{\frac{p-1}{2p}}\|v_i\|^{-\frac{p-1}{2p}}. 
\end{equation*}
This implies
\begin{equation*}
\left(\theta^{\frac{1}{2p}}(1-\sigma)^{\frac{p-1}{2p}}\right)\inf_{1 \leq i \leq k} \|v_i\|^{\frac{p+1}{2p}} \leq \inf_{1 \leq i \leq k} \sqrt{\lambda_i} \|v_i\| = O(k^{-\frac{p+1}{4}}) \Longrightarrow \inf_{1 \leq i \leq k} \|v_i\|^{\frac{p+1}{2p}} = O(k^{-\frac{p+1}{4}}). 
\end{equation*}
Therefore, we conclude that 
\begin{equation*}
\inf_{1 \leq i \leq k} \|v_i\| = \left(\inf_{1 \leq i \leq k} \|v_i\|^{\frac{p+1}{2p}}\right)^{\frac{2p}{p+1}} = O(k^{-\frac{p}{2}}). 
\end{equation*}
It remains to prove that the iterates $\{x_k\}_{k \geq 1}$ converge to $A^{-1}(0)$ with a local linear rate under the error bound condition in Eq.~\eqref{def:EB} and that $\epsilon_k = 0$ for all $k \geq 1$.  Indeed,  it follows from the proof of Lemma~\ref{Lemma:CAF-descent} that 
\begin{equation*}
\ECal_k - \ECal_{k+1} \geq \lambda_{k+1} \langle v,  y_{k+1} - z\rangle + \tfrac{1 - \sigma^2}{2}\|x_k - y_{k+1}\|^2, \quad \textnormal{for all } v \in Az. 
\end{equation*}
Recall that $\ECal_k = \frac{1}{2}\|x_k - z\|^2$. Thus, we have
\begin{equation*}
\|x_k - z\|^2 - \|x_{k+1} - z\|^2 \geq 2\lambda_{k+1} \langle v,  y_{k+1} - z\rangle + (1 - \sigma^2)\|x_k - y_{k+1}\|^2, \quad \textnormal{for all } v \in Az. 
\end{equation*}
Here $z \in \Br^d$ can be any point.  Then, we set $z = x_k^\star = \textnormal{argmin}_{x \in A^{-1}(0)} \|x - x_k\|$ and choose $v = 0 \in Az$. Plugging into the above inequality implies 
\begin{equation*}
\|x_k - x_k^\star\|^2 - \|x_{k+1} - x_k^\star\|^2 \geq 2\lambda_{k+1} \langle 0, y_{k+1} - x_k^\star\rangle + (1-\sigma^2)\|x_k - y_{k+1}\|^2 = (1-\sigma^2)\|x_k - y_{k+1}\|^2.  
\end{equation*}
By definition, we have $\|x_{k+1} - x_{k+1}^\star\| \leq \|x_{k+1} - x_k^\star\|$ and $\textsc{dist}(x_k, A^{-1}(0)) = \|x_k - x_k^\star\|$. Putting these pieces together yields that, for all $k \geq 1$, we have 
\begin{equation}\label{inequality:CAF-main-first}
(\textsc{dist}(x_k, A^{-1}(0)))^2 - (\textsc{dist}(x_{k+1}, A^{-1}(0)))^2 \geq (1-\sigma^2)\|x_k - y_{k+1}\|^2. 
\end{equation}
It is worth mentioning that Eq.~\eqref{inequality:CAF-main-first} implies that $\|x_k - y_{k+1}\| \rightarrow 0$.  Using the large step condition that $\lambda_k\|y_k - x_{k-1}\|^{p-1} \geq \theta$, we have $\{\lambda_k\}_{k \geq 1}$ is lower bounded by a constant $\underline{\lambda} > 0$.  Further, we have
\begin{equation*}
\|\lambda_k v_k\| \leq \|\lambda_k v_k + y_k - x_{k-1}\| + \|y_k - x_{k-1}\| \leq (1+\sigma)\|y_k - x_{k-1}\|, 
\end{equation*}
which implies that $\|v_k\| \rightarrow 0$ as $k \rightarrow +\infty$.  Since $\epsilon_k = 0$ for all $k \geq 1$, we have $v_k \in Ay_k$.  So there exists a sufficiently large $k_0 > 0$ such that $\textsc{dist}(0, Ay_k) \leq \delta$ for all $k \geq k_0$ where $\delta > 0$ is defined in the error bound condition (cf. Eq.~\eqref{def:EB}). Since the error bound condition in Eq.~\eqref{def:EB} holds true, we have
\begin{equation*}
\textsc{dist}(y_{k+1}, A^{-1}(0)) \leq \kappa \cdot \textsc{dist}(0, Ay_{k+1}) \leq \kappa \|v_{k+1}\|.  
\end{equation*}
We let $y_{k+1}^\star = \argmin_{y \in A^{-1}(0)} \|y - y_{k+1}\|$ and deduce from the triangle inequality that 
\begin{equation*}
\textsc{dist}(x_k, A^{-1}(0)) \leq \|x_k - y_{k+1}^\star\| \leq \|x_k - y_{k+1}\| + \textsc{dist}(y_{k+1}, A^{-1}(0)) \leq \|x_k - y_{k+1}\| + \kappa \|v_{k+1}\|.
\end{equation*}
Putting these pieces together yields that
\begin{equation}\label{inequality:CAF-main-second}
\textsc{dist}(x_k, A^{-1}(0)) \leq \left(1 + \tfrac{\kappa}{\lambda_{k+1}}\right)\|y_{k+1} - x_k\| \leq \left(1 + \tfrac{\kappa(1 + \sigma)}{\underline{\lambda}}\right)\|y_{k+1} - x_k\|. 
\end{equation} 
Plugging Eq.~\eqref{inequality:CAF-main-second} into Eq.~\eqref{inequality:CAF-main-first} yields that 
\begin{equation*}
(\textsc{dist}(x_k, A^{-1}(0)))^2 - (\textsc{dist}(x_{k+1}, A^{-1}(0)))^2 \geq (1-\sigma^2)\left(\tfrac{\underline{\lambda}}{\kappa(1+\sigma) + \underline{\lambda}}\right)^2(\textsc{dist}(x_k, A^{-1}(0)))^2. 
\end{equation*}
This completes the proof. 
\begin{remark}
The discrete-time analysis in Theorem~\ref{Theorem:CAF-main} is based on the Lyapunov function from Eq.~\eqref{def:Lyapunov-discrete}, which is inspired by the one in Eq.~\eqref{def:Lyapunov} and Eq.~\eqref{def:EB-Lyapunov}. Notably, the proofs of these technical results follow the same path for the continuous-time analysis in Theorem~\ref{Theorem:Trajectory-Convergence-Rate} and~\ref{Theorem:Trajectory-Convergence-Linear}. 
\end{remark}

\subsection{Global acceleration and local linear convergence}
By instantiating Algorithm~\ref{Algorithm:CAF} with approximate tensor subroutines~\citep{Nesterov-2021-Implementable}, we develop a new family of $p^\textnormal{th}$-order tensor algorithms for monotone inclusion problems with $A = F + H$ in which $F \in \GCal_L^p(\br^d)$ is a point-to-point operator and $H$ is simple and maximal monotone. We provide new convergence results concerning these tensor algorithms, including an ergodic rate of $O(k^{-(p+1)/2})$ in terms of a gap function, a pointwise rate of $O(k^{-p/2})$ in terms of a residue function, and establish local linear convergence under an error bound condition. Our results extend those results in~\citet{Monteiro-2012-Iteration} for second-order algorithms for monotone inclusion problems and complement the analysis in~\citet{Bullins-2022-Higher} concerning high-order tensor algorithms for saddle point and variational inequality problems.  

The proximal point algorithm (PPA) (corresponding to implicit discretization of certain systems) requires solving an exact proximal iteration with proximal coefficient $\lambda > 0$ at each iteration:   
\begin{equation}\label{subprob:prox-exact}
y = (I + \lambda (F + H))^{-1}(x).  
\end{equation}
In many application problem, $H = \partial \textbf{1}_\XCal$, where $\partial \textbf{1}_\XCal$ is the subdifferential of an indicator function onto a closed and convex set $\XCal$.  Nevertheless, Eq.~\eqref{subprob:prox-exact} is still hard when the proximal coefficient $\lambda \rightarrow +\infty$. Fortunately, when $F \in \GCal_L^p(\br^d)$, it suffices to solve the subproblem with the $(p-1)^\textnormal{th}$-order approximation of $F$. More specifically, we define   
\begin{equation}\label{def:approximation}
F_x(u) = F(x) + \langle DF(x), u-x\rangle + \sum_{j=2}^{p-1} \tfrac{1}{j!} D^{(j)} F(x)[u-x]^j.   
\end{equation}
Our algorithms of this subsection are based on an inexact solution of the following subproblem: 
\begin{equation}\label{subprob:inexact}
y = (I + \lambda (F_x + H))^{-1}(x). 
\end{equation}
Clearly, the solution $x_v$ of Eq.~\eqref{subprob:inexact} is unique and satisfies $\lambda F_x(x_v) + H(x_v) + x_v - x = 0$. Thus, we denote a \textit{$\hat{\sigma}$-inexact solution} of Eq.~\eqref{subprob:inexact} at $(\lambda, x)$ by a vector $y \in \br^d$ satisfying that $u \in (F_{x'} + H)(y)$ and $\|\lambda u + y - x\| \leq \hat{\sigma}\|y - x\|$ for some $\hat{\sigma} \in (0, 1)$ and $x' = \proj_{\dom(A)}(x)$ (recall $A = F+H$).  
\begin{algorithm}[!t]
\begin{algorithmic}\caption{Accelerated $p^\textnormal{th}$-order Tensor Algorithm}\label{Algorithm:Optimal}
\STATE \textbf{STEP 0:}  Let $x_0 \in \br^d$, $\hat{\sigma} \in (0, 1)$ and $0 < \sigma_l < \sigma_u < 1$ such that $\sigma_l(1+\hat{\sigma})^{p-1} < \sigma_u(1-\hat{\sigma})^{p-1}$ and $\sigma = \hat{\sigma} + \sigma_u < 1$ be given, and set $k=0$. 
\STATE \textbf{STEP 1:} Compute $x'_k = \proj_{\dom(A)}(x_k)$. If $0 \in A(x_k)$, then \textbf{stop}. 
\STATE \textbf{STEP 2:} Otherwise, compute a positive scalar $\lambda_{k+1} > 0$ with a $\hat{\sigma}$-inexact solution $y_{k+1} \in \br^d$ of Eq.~\eqref{subprob:inexact} at $(\lambda_{k+1}, x_k)$ satisfying that
\begin{equation*}
u_{k+1} \in (F_{x'_k} + H)(y_{k+1}), \quad \|\lambda_{k+1}u_{k+1} + y_{k+1} - x_k\| \leq \hat{\sigma}\|y_{k+1} - x_k\|,  
\end{equation*}
and
\begin{equation*}
\tfrac{\sigma_l p!}{L} \leq \lambda_{k+1}\|y_{k+1} - x_k\|^{p-1} \leq \tfrac{\sigma_u p!}{L}. 
\end{equation*}
\STATE \textbf{STEP 3:} Compute $v_{k+1} = F(y_{k+1}) + u_{k+1} - F_{x'_k}(y_{k+1})$. 
\STATE \textbf{STEP 4:} Compute $x_{k+1} = x_k - \lambda_{k+1}v_{k+1}$. 
\STATE \textbf{STEP 5:}  Set $k \leftarrow k+1$, and go to \textbf{STEP 1}. 
\end{algorithmic}
\end{algorithm}

We summarize our accelerated $p^\textnormal{th}$-order tensor algorithm in Algorithm~\ref{Algorithm:Optimal} and prove that it is an application of Algorithm~\ref{Algorithm:CAF} with a specific choice of $\theta$. 
\begin{proposition}\label{Prop:opt}
Algorithm~\ref{Algorithm:Optimal} is Algorithm~\ref{Algorithm:CAF} with $\theta = \frac{\sigma_l p!}{L}$ and $\epsilon_k = 0$ for all $k \geq 1$. 
\end{proposition}
\begin{proof}
Letting a tuple $(x_k, v_k, u_k)_{k \geq 1}$ be generated by Algorithm~\ref{Algorithm:Optimal},  it is clear that 
\begin{equation*}
v_{k+1} = F(y_{k+1}) + u_{k+1} - F_{x'_k}(y_{k+1}) \in (F + H)(y_{k+1}). 
\end{equation*}
This is equivalent to that $v_{k+1} \in A^{\epsilon_{k+1}}(y_{k+1})$ in Algorithm~\ref{Algorithm:CAF} with $\varepsilon_k = 0$ for all $k \geq 1$.  Further, since $\theta = \frac{\sigma_l p!}{L} > 0$, we have
\begin{equation*}
\lambda_{k+1}\|y_{k+1} - x_k\|^{p-1} \geq \tfrac{\sigma_l p!}{L} \Longrightarrow \lambda_{k+1}\|y_{k+1} - x_k\|^{p-1} \geq \theta. 
\end{equation*}
It suffices to show that $\|\lambda_{k+1}v_{k+1} + y_{k+1} - x_k\|^2 + 2\lambda_{k+1}\epsilon_{k+1} \leq \sigma^2\|y_{k+1} - x_k\|^2$. Since $\epsilon_k = 0$ for all $k \geq 1$, the above inequality is equivalent to 
\begin{equation}\label{inequality:opt-first}
\|\lambda_{k+1}v_{k+1} + y_{k+1} - x_k\| \leq \sigma\|y_{k+1} - x_k\|. 
\end{equation}
Indeed, we have
\begin{eqnarray*}
\|\lambda_{k+1}v_{k+1} + y_{k+1} - x_k\| & \leq & \lambda_{k+1}\|v_{k+1} - u_{k+1}\| + \|\lambda_{k+1}u_{k+1} + y_{k+1} - x_k\| \\ 
& \leq & \lambda_{k+1}\|F(y_{k+1}) - F_{x'_k}(y_{k+1})\| + \hat{\sigma}\|y_{k+1} - x_k\|. 
\end{eqnarray*}
Using the definition of $F_{x'_k}$ (cf. Eq.~\eqref{def:approximation}) and the fact that $F \in \GCal_L^p(\br^d)$, we have
\begin{equation*}
\lambda_{k+1}\|F(y_{k+1}) - F_{x'_k}(y_{k+1})\| \leq \tfrac{\lambda_{k+1}L}{p!}\|y_{k+1} - x'_k\|^p.
\end{equation*}
Since $y_{k+1} \in \dom(A)$ and $x'_k = \proj_{\dom(A)}(x_k)$, we have $\|y_{k+1} - x'_k\| \leq \|y_{k+1} - x_k\|$. This implies
\begin{equation}\label{inequality:opt-second}
\lambda_{k+1}\|F(y_{k+1}) - F_{x'_k}(y_{k+1})\| \leq \tfrac{\lambda_{k+1}L}{p!}\|y_{k+1} - x_k\|^p.
\end{equation}
Since $\sigma = \hat{\sigma} + \sigma_u < 1$, we have
\begin{equation}\label{inequality:opt-third}
\lambda_{k+1}\|y_{k+1} - x_k\|^{p-1} \leq \tfrac{\sigma_u p!}{L} \Longrightarrow \hat{\sigma} + \tfrac{\lambda_{k+1}L}{p!}\|y_{k+1}-x_k\|^{p-1} \leq \hat{\sigma} + \sigma_u = \sigma. 
\end{equation}
Combing Eq.~\eqref{inequality:opt-second} and Eq.~\eqref{inequality:opt-third}, we have
\begin{equation*}
\lambda_{k+1}\|F(y_{k+1}) - F_{x'_k}(y_{k+1})\| + \hat{\sigma}\|y_{k+1} - x_k\| \leq \left(\tfrac{\lambda_{k+1}L}{p!}\|y_{k+1}-x_k\|^{p-1} + \hat{\sigma}\right)\|y_{k+1} - x_k\| \leq \sigma\|y_{k+1} - x_k\|. 
\end{equation*}
Putting these pieces together yields the desired equation in Eq.~\eqref{inequality:opt-first}.  
\end{proof}
In view of Proposition~\ref{Prop:opt}, the iteration complexity derived for Algorithm~\ref{Algorithm:CAF} holds for Algorithm~\ref{Algorithm:Optimal}. Furthermore, we have $v_k \in (F + H)(y_k) = Ay_k$ for all $k \geq 1$ which implies (see the definition of a residue function in Eq.~\eqref{def:residue-function})
\begin{equation*}
\textsc{res}(y_k) = \inf_{\xi \in Ay_k} \|\xi\| \leq \|v_k\|. 
\end{equation*}
As a consequence of Theorem~\ref{Theorem:CAF-main}, we summarize the results in the following theorem. 
\begin{theorem}\label{Theorem:Optimal-main}
For every integer $k \geq 1$ and let $\dom(A)$ be closed and bounded. Then,  we have
\begin{equation*}
\textsc{gap}(\tilde{y}_k) = O(k^{-\frac{p+1}{2}}), 
\end{equation*}
and 
\begin{equation*}
\textsc{res}(y_k) = O(k^{-\frac{p}{2}}), 
\end{equation*}
where the ergodic iterates $\{\tilde{y}_k\}_{k \geq 1}$ are defined by 
\begin{equation*}
\tilde{y}_k = \tfrac{1}{\sum_{i=1}^k \lambda_k}\left(\sum_{i=1}^k \lambda_i y_i\right). 
\end{equation*}
In addition, if we assume that the error bound condition in Eq.~\eqref{def:EB} holds true, the iterates $\{x_k\}_{k \geq 1}$ converge to $A^{-1}(0)$ with a local linear rate. 
\end{theorem}
\begin{remark}
The ergodic and pointwise convergence results in Theorem~\ref{Theorem:Optimal-main} have been obtained in~\citet[Theorem~3.5 and~3.6]{Monteiro-2012-Iteration} for the case of $p=2$ and derived by~\citet{Nemirovski-2004-Prox} and~\citet{Monteiro-2010-Complexity} for the extragradient method (the case of $p=1$). For $p \geq 3$ in general, these global convergence results generalize~\citet[Theorem~4.5]{Bullins-2022-Higher} from saddle point and variational inequality problems to monotone inclusion problems. The local linear convergence results under an error bound are well known for the extragradient method in the literature~\citep{Tseng-1995-Linear,Monteiro-2010-Complexity} but are new for the case of $p \geq 2$ to our knowledge. 
\end{remark}
\begin{remark}
The approximate tensor subroutine in Algorithm~\ref{Algorithm:Optimal} can be efficiently implemented using binary search procedures specialized to the case of $p=2$; see~\citet[Section~4]{Monteiro-2012-Iteration} and~\citet[Section~5]{Bullins-2022-Higher}.  Could we generalize this scheme to handle the more general case of $p \geq 3$, similar to what has been accomplished in convex optimization~\citep{Gasnikov-2019-Near}? We leave the answer to this question to future work.
\end{remark}

\section{Conclusions}\label{sec:conclusions}
We propose a new closed-loop control system for capturing the acceleration phenomenon in monotone inclusion problems. In terms of theoretical guarantee, we obtain ergodic and pointwise convergence rates via appeal to simple and intuitive Lyapunov functions. Our framework based on implicit discretization of the aforementioned system gives a systematic way to derive  discrete-time $p^\textnormal{th}$-order accelerated tensor algorithms for all $p \geq 1$ and simplify existing analyses via the use of a discrete-time Lyapunov function. Key to our framework is the algebraic equation, which disappears for the case of $p = 1$, but is essential for achieving the acceleration for the case of $p \geq 2$. We also infer that a certain class of $p^\textnormal{th}$-order tensor algorithms can achieve local linear convergence under an error bound condition. 

It is worth mentioning that our closed-loop control system is related to the nonlinear damping in the PDE literature where the closed-loop feedback control in fact depends on the velocity~\citep{Attouch-2022-Fast}; indeed, it is demonstrated by the algebraic equation $\lambda(t)\|\dot{x}(t)\|^{p-1} = \theta$.

There are several other avenues for future research.  For example, it is interesting to study the monotone inclusion problems via appeal to the Lagrangian and Hamiltonian frameworks that have proved productive in recent work~\citep{Wibisono-2016-Variational,Diakonikolas-2021-Generalized,Muehlebach-2021-Optimization,Francca-2021-Dissipative}. Moreover, we would hope for this study to provide additional insight into the geometric or dynamical role played by the algebraic equation for shaping the continuous-time dynamics. Indeed, it is of interest to investigate the continuous-time limit of Newton methods for Bouligand-differentiable equations~\citep{Robinson-1987-Local},  which is another generalization of complementarity and VI problems, and see whether the closed-loop control approach leads to efficient algorithms or not.   

\section*{Acknowledgments}
The authors would like to thank the associate editor and the reviewer for their extremely valuable comments and constructive feedback that led to significant improvement of the current paper. This work was supported in part by the Mathematical Data Science program of the Office of Naval Research under grant number N00014-18-1-2764 and by the Vannevar Bush Faculty Fellowship program
under grant number N00014-21-1-2941.

\bibliographystyle{plainnat}
\bibliography{ref}

\appendix
\section{Proof of Lemma~\ref{Lemma:AE-mapping-Lipschitz}}
For simplicity, we denote by $A_\lambda = I - (I + \lambda A)^{-1}$ and write $\varphi(\lambda, x) = \lambda^{1/(p-1)}\|A_\lambda x\|$. Then it suffices to show
\begin{equation}\label{inequality:AE-mapping-Lipschitz-main}
\left|\|A_\lambda x_1\| - \|A_\lambda x_2\|\right| \leq \|x_1 - x_2\|. 
\end{equation}
It is known in convex analysis (see~\cite{Rockafellar-1970-Convex} for example) that $\|A_\lambda x_1 - A_\lambda x_2\| \leq \|x_1 - x_2\|$. This together with the triangle inequality yields Eq.~\eqref{inequality:AE-mapping-Lipschitz-main}. 

\section{Proof of Lemma~\ref{Lemma:AE-mapping-monotone}}
For simplicity, we define  $z_1 = (I + \lambda_1 A)^{-1}x$, $z_2 = (I + \lambda_2 A)^{-1}x$, $v_1 \in Az_1$ and $v_2 \in Az_2$. In view of the definitions, we have
\begin{equation}\label{inequality:AE-mapping-monotone-first}
\begin{array}{lcl}
\varphi(\lambda_1, x) = (\lambda_1)^{\frac{1}{p-1}}\|x - z_1\|, & & \lambda_1 v_1 + z_1 - x = 0,  \\ 
\varphi(\lambda_2, x) = (\lambda_2)^{\frac{1}{p-1}}\|x - z_2\|, & & \lambda_2 v_2 + z_2 - x = 0.  
\end{array}
\end{equation}
After straightforward calculation, we have
\begin{equation*}
\lambda_1(v_1 - v_2) + z_1 - z_2 = (\lambda_1 v_1 + z_1 - x) - (\lambda_2 v_2 + z_2 - x) + (\lambda_2 - \lambda_1)v_2 = (\lambda_2 - \lambda_1)v_2. 
\end{equation*}
Since $A$ is a maximal monotone operator, $v_1 \in Az_1$ and $v_2 \in Az_2$, we have $\langle v_1 - v_2, z_1 - z_2\rangle \geq 0$.  Putting these pieces together yields
\begin{equation*}
(\lambda_2 - \lambda_1)\langle v_1 - v_2, v_2 \rangle = \lambda_1\|v_1 - v_2\|^2 + \langle v_1 - v_2, z_1 - z_2\rangle \geq 0. 
\end{equation*}
This together with $\lambda_1 \leq \lambda_2$ implies that $\langle v_1 - v_2, v_2 \rangle \geq 0$ and thus we have $\|v_1\| \geq \|v_2\|$.  Combining the last inequality with Eq.~\eqref{inequality:AE-mapping-monotone-first}, we have
\begin{equation}\label{inequality:AE-mapping-monotone-second}
\varphi(\lambda_2, x) = (\lambda_2)^{\frac{1}{p-1}}\|\lambda_2 v_2\| = (\lambda_2)^{\frac{p}{p-1}}\|v_2\| \leq (\lambda_2)^{\frac{p}{p-1}}\|v_1\| = \left(\tfrac{\lambda_2}{\lambda_1}\right)^{\frac{p}{p-1}}\varphi(\lambda_1, x). 
\end{equation}
After a short calculation, we have
\begin{equation*}
v_1 - v_2 + \tfrac{1}{\lambda_2}(z_1 - z_2) = \tfrac{1}{\lambda_1}(\lambda_1 v_1 + z_1 - x) - \tfrac{1}{\lambda_2}(\lambda_2 v_2 + z_2 - x) + \left(\tfrac{1}{\lambda_2} - \tfrac{1}{\lambda_1}\right)(z_1 - x) = \left(\tfrac{1}{\lambda_2} - \tfrac{1}{\lambda_1}\right)(z_1 - x).  
\end{equation*}
Since $\langle v_1 - v_2, z_1 - z_2\rangle \geq 0$, we have
\begin{equation*}
\left(\tfrac{1}{\lambda_2} - \tfrac{1}{\lambda_1}\right)\langle z_1 - x, z_1 - z_2\rangle = \langle v_1 - v_2, z_1 - z_2\rangle + \tfrac{1}{\lambda_2}\|z_1 - z_2\|^2 \geq 0. 
\end{equation*}
This together with $\lambda_1 \leq \lambda_2$ implies that $\langle z_1 - x, z_1 - z_2\rangle \leq 0$ and thus we have $\|x - z_2\| \geq \|x - z_1\|$.  Combining the last inequality with Eq.~\eqref{inequality:AE-mapping-monotone-first}, we have
\begin{equation}\label{inequality:AE-mapping-monotone-third}
\varphi(\lambda_2, x) = (\lambda_2)^{\frac{1}{p-1}}\|x - z_2\| \geq (\lambda_2)^{\frac{1}{p-1}}\|x - z_1\| = \left(\tfrac{\lambda_2}{\lambda_1}\right)^{\frac{1}{p-1}}\varphi(\lambda_1, x). 
\end{equation}
Combining Eq.~\eqref{inequality:AE-mapping-monotone-second} and Eq.~\eqref{inequality:AE-mapping-monotone-third} yields the desired inequality.  The last statement of the lemma follows trivially from the maximal monotonicity of $A$ and the definition of $\varphi$. 

\section{Proof of Lemma~\ref{Lemma:control-Lipschitz}}
Rearranging the first inequality in Lemma~\ref{Lemma:AE-mapping-monotone} implies
\begin{equation*}
\tfrac{\varphi(\lambda_1, x)}{(\lambda_1)^{\frac{1}{p-1}}} \leq \tfrac{\varphi(\lambda_2, x)}{(\lambda_2)^{\frac{1}{p-1}}}, \quad \textnormal{for any } x \in \HCal \textnormal{ and } 0 < \lambda_1 \leq \lambda_2. 
\end{equation*}
This yields that the mapping $\lambda \mapsto \|x - (I + \lambda A)^{-1}x\|$ is nondecreasing and $\alpha \mapsto \|x - (I + \alpha^{-1} A)^{-1}x\|$ is a (continuous) nonincreasing function. As a consequence, we obtain that $\Gamma_\theta$ is a real-valued nonnegative function.  Further, by definition, if $x \in \Omega$, we have
\begin{eqnarray*}
\Gamma_\theta(x) \ = \ \left(\inf\{ \alpha > 0 \mid \alpha^{-\frac{1}{p-1}}\|x - (I + \alpha^{-1}A)^{-1}x\| \leq \theta^{\frac{1}{p-1}}\}\right)^{\frac{1}{p-1}} = \left(\tfrac{1}{\Lambda_\theta(x)}\right)^{\frac{1}{p-1}}. 
\end{eqnarray*}
Moreover, if $x \notin \Omega$,  we have $x - (I + \alpha^{-1}A)^{-1}x = 0$ for all $\alpha > 0$ which implies that $\Gamma_\theta(x) = 0$.  Putting these pieces together yields the desired relationship between $\Gamma_\theta(x)$ and $\Lambda_\theta(x)$. 

It remains to show that $\Gamma_\theta: \HCal \mapsto (0, +\infty)$ is Lipschitz continuous with $\theta^{-1/(p-1)} > 0$.  Take $x_1, x_2 \in \HCal$ and $\bar{\alpha} > 0$, we suppose that
\begin{equation*}
\|x_1 - (I + \bar{\alpha}^{-1}A)^{-1}x_1\| \leq \bar{\alpha}^{\frac{1}{p-1}} \theta^{\frac{1}{p-1}}. 
\end{equation*}
It is straightforward to deduce that $I - (I + \bar{\alpha}^{-1}A)^{-1}$ is 1-Lipschitz continuous (see~\cite{Rockafellar-1970-Convex} or the proof of Lemma~\ref{Lemma:AE-mapping-Lipschitz}) which implies
\begin{equation*}
\|x_2 - (I + \bar{\alpha}^{-1}A)^{-1}x_2\| - \|x_1 - (I + \bar{\alpha}^{-1}A)^{-1}x_1\| \leq \|x_2 - x_1\|. 
\end{equation*}
Therefore, we have
\begin{equation*}
\|x_2 - (I + \bar{\alpha}^{-1}A)^{-1}x_2\| \leq \bar{\alpha}^{\frac{1}{p-1}} \theta^{\frac{1}{p-1}} + \|x_2 - x_1\| = (\bar{\alpha}^{\frac{1}{p-1}} + \|x_2 - x_1\|\theta^{-\frac{1}{p-1}})\theta^{\frac{1}{p-1}}. 
\end{equation*}
Let $\bar{\beta} = (\bar{\alpha}^{1/(p-1)} + \|x_2 - x_1\|\theta^{-1/(p-1)})^{p-1}$. Since $\bar{\beta} \geq \bar{\alpha}$ and the mapping $\alpha \mapsto \|x - (I + \alpha^{-1} A)^{-1}x\|$ is nonincreasing,  we have
\begin{equation*}
\|x_2 - (I + \bar{\beta}^{-1}A)^{-1}x_2\| \leq \|x_2 - (I + \bar{\alpha}^{-1}A)^{-1}x_2\| \leq \bar{\beta}^{\frac{1}{p-1}}\theta^{\frac{1}{p-1}}. 
\end{equation*}
By the definition of $\Gamma_\theta$, we have $\Gamma_\theta(x_2) \leq \bar{\beta}^{1/(p-1)} = \bar{\alpha}^{1/(p-1)} + \|x_2 - x_1\|\theta^{-1/(p-1)}$ and this inequality holds true for all $\bar{\alpha} > 0$ satisfying that $\|x_1 - (I + \bar{\alpha}^{-1}A)^{-1}x_1\| \leq \bar{\alpha}^{1/(p-1)}\theta^{1/(p-1)}$; that is, all $\bar{\alpha} > 0$ satisfying that $\bar{\alpha}^{1/(p-1)} \geq \Gamma_\theta(x_1)$. Putting these pieces together, we have
\begin{equation*}
\Gamma_\theta(x_2) \leq \Gamma_\theta(x_1) + \|x_2 - x_1\|\theta^{-\frac{1}{p-1}}. 
\end{equation*} 
Using the symmetry of $x_1$ and $x_2$, we have $\Gamma_\theta(x_1) \leq \Gamma_\theta(x_2) + \|x_1 - x_2\|\theta^{-1/(p-1)}$. Therefore, we have
\begin{equation*}
|\Gamma_\theta(x_1) - \Gamma_\theta(x_2)| \leq \theta^{-\frac{1}{p-1}}\|x_1 - x_2\|. 
\end{equation*}
This completes the proof. 

\section{Proof of Lemma~\ref{Lemma:control-first}}
For the case of $p = 1$, we have $\lambda(t) = \theta$ is a constant function and the desired result holds true.  For the case of $p \geq 2$, let $t, t' \in [0, t_0]$ and $t \neq t'$. Then, we have
\begin{equation*}
\left|(\lambda(t'))^{\frac{1}{p-1}} - (\lambda(t))^{\frac{1}{p-1}}\right| = (\lambda(t')\lambda(t))^{\frac{1}{p-1}}\left|\left(\tfrac{1}{\lambda(t)}\right)^{\frac{1}{p-1}} - \left(\tfrac{1}{\lambda(t')}\right)^{\frac{1}{p-1}}\right|. 
\end{equation*}
Using the definition of $\Gamma_\theta(\cdot)$ and the fact that it is Lipschitz continuous with a constant $\theta^{-1/(p-1)} > 0$ (cf. Lemma~\ref{Lemma:control-Lipschitz}), we have
\begin{equation*}
\left|\left(\tfrac{1}{\lambda(t)}\right)^{\frac{1}{p-1}} - \left(\tfrac{1}{\lambda(t')}\right)^{\frac{1}{p-1}}\right| = |\Gamma_\theta(x(t)) - \Gamma_\theta(x(t'))| \leq \tfrac{\|x(t) - x(t')\|}{\theta^{1/(p-1)}}. 
\end{equation*}
Putting these pieces together yields 
\begin{equation*}
\left|\tfrac{(\lambda(t'))^{\frac{1}{p-1}} - (\lambda(t))^{\frac{1}{p-1}}}{t' - t}\right| \leq \left(\tfrac{\lambda(t')\lambda(t)}{\theta}\right)^{\frac{1}{p-1}}\tfrac{\|x(t) - x(t')\|}{|t - t'|}.
\end{equation*}
Fix $t \in [0, t_0]$ and let $t' \rightarrow t$. Then, we have
\begin{equation*}
\limsup_{t' \rightarrow t} \left|\tfrac{(\lambda(t'))^{\frac{1}{p-1}} - (\lambda(t))^{\frac{1}{p-1}}}{t' - t}\right| \leq \left(\tfrac{(\lambda(t))^2}{\theta}\right)^{\frac{1}{p-1}}\|\dot{x}(t)\|. 
\end{equation*}
Using Eq.~\eqref{sys:general} and Eq.~\eqref{sys:choice-feedback}, we have
\begin{equation*}
\|\dot{x}(t)\| = \|(I + \lambda(t)A)^{-1}x(t) - x(t)\| = \left(\tfrac{\theta}{\lambda(t)}\right)^{\frac{1}{p-1}}. 
\end{equation*}
In addition, for almost all $t \in [0, t_0]$, we have
\begin{equation*}
\limsup_{t' \rightarrow t} \left|\tfrac{(\lambda(t'))^{\frac{1}{p-1}} - (\lambda(t))^{\frac{1}{p-1}}}{t' - t}\right| \geq \tfrac{1}{p-1}|\dot{\lambda}(t)(\lambda(t))^{\frac{1}{p-1}-1}|. 
\end{equation*}
Putting these pieces together yields the desired result. 

\section{Proof of Lemma~\ref{Lemma:control-second}}
For the case of $p = 1$, we have $\lambda(t) = \theta$ is a constant function and the desired result holds true.  For the case of $p \geq 2$,  since $\lambda(\cdot)$ is locally Lipschitz continuous, it suffices to show that $\dot{\lambda}(t) \geq 0$ for almost all $t \in [0, t_0)$. Indeed, let $y(t) = (I + \lambda(t)A)^{-1}x(t)$,  we deduce from Eq.~\eqref{sys:general} that $\dot{x}(t) = y(t) - x(t)$. Then, for any fixed $t \in (0, t_0)$, we have $0 < h < \min\{t_0 - t, 1\}$ exists and the following inequality holds: 
\begin{equation*}
x(t) + h\dot{x}(t) = (1- h)x(t) + hy(t) \Longrightarrow x(t) + h\dot{x}(t) - y(t) = (1- h)(x(t) - y(t)).  
\end{equation*}
By the definition of $y(\cdot)$, we have $\frac{1}{\lambda(t)}(x(t) - y(t)) \in Ay(t)$.  Combining this with the above equality yields that $y(t) = (I + (1-h)\lambda(t)A)^{-1}(x(t) + h\dot{x}(t))$.  Then, by the definition of $\varphi$, we have
\begin{eqnarray*}
\lefteqn{ \varphi((1-h)\lambda(t), x(t) + h\dot{x}(t))} \\
& = & (1-h)^{\frac{1}{p-1}}(\lambda(t))^{\frac{1}{p-1}}\|x(t) + h\dot{x}(t) - (I + (1-h)\lambda(t)A)^{-1}(x(t) + h\dot{x}(t))\| \\
& = & (1-h)^{\frac{1}{p-1}}(\lambda(t))^{\frac{1}{p-1}}\|x(t) + h\dot{x}(t) - y(t)\| \\
& = & (1-h)^{\frac{p}{p-1}}(\lambda(t))^{\frac{1}{p-1}}\|x(t) - y(t)\|. 
\end{eqnarray*}
In addition, Eq.~\eqref{sys:choice-feedback} implies that $(\lambda(t))^{1/(p-1)}\|x(t) - y(t)\| = \theta^{1/(p-1)}$. Putting these pieces together yields
\begin{equation}\label{inequality:control-first}
\varphi((1-h)\lambda(t), x(t) + h\dot{x}(t)) = (1-h)^{\frac{p}{p-1}}\theta^{\frac{1}{p-1}}. 
\end{equation}
Using the triangle inequality and Lemma~\ref{Lemma:AE-mapping-Lipschitz}, we have
\begin{eqnarray}\label{inequality:control-second}
\varphi(\lambda(t), x(t+h)) & \leq & \varphi(\lambda(t), x(t) + h\dot{x}(t)) + |\varphi(\lambda(t), x(t+h)) - \varphi(\lambda(t), x(t) + h\dot{x}(t))| \nonumber \\
& \leq & \varphi(\lambda(t), x(t) + h\dot{x}(t)) + (\lambda(t))^{\frac{1}{p-1}}\| x(t+h) - x(t) - h\dot{x}(t)\|. 
\end{eqnarray}
Using the second inequality in Lemma~\ref{Lemma:AE-mapping-monotone} and $0 < h < 1$, we have
\begin{equation}\label{inequality:control-third}
\varphi(\lambda(t), x(t) + h\dot{x}(t)) \leq \left(\tfrac{1}{1-h}\right)^{\frac{p}{p-1}}\varphi((1-h)\lambda(t), x(t) + h\dot{x}(t)) \overset{\textnormal{Eq.~\eqref{inequality:control-first}}}{=} \theta^{\frac{1}{p-1}}. 
\end{equation}
For the ease of presentation, we define the function $\omega: (0, \min\{t_0 - t, 1\}) \mapsto (0, +\infty)$ by 
\begin{equation*}
\omega(h) = \left(\tfrac{\lambda(t)\|x(t+h) - x(t) - h\dot{x}(t)\|^{p-1}}{\theta}\right)^{\frac{1}{p-1}}. 
\end{equation*}
Plugging Eq.~\eqref{inequality:control-third} into Eq.~\eqref{inequality:control-second} and simplifying the resulting inequality using the definition of $\omega(\cdot)$ yields
\begin{equation*}
\varphi(\lambda(t), x(t+h)) \leq \theta^{\frac{1}{p-1}}\left(1 + \omega(h)\right). 
\end{equation*}
Using the first inequality in Lemma~\ref{Lemma:AE-mapping-monotone} and $\omega(h) \geq 0$ for all $h \in (0, \min\{t_0 - t, 1\})$, we have
\begin{equation*}
\varphi\left(\tfrac{\lambda(t)}{(1+\omega(h))^{p-1}}, x(t+h)\right) \leq \left(\tfrac{1}{(1+\omega(h))^{p-1}}\right)^{\frac{1}{p-1}}\varphi(\lambda(t), x(t+h)). 
\end{equation*}
Putting these pieces together yields
\begin{equation*}
\varphi\left(\tfrac{\lambda(t)}{(1+\omega(h))^{p-1}}, x(t+h)\right) \leq \theta^{\frac{1}{p-1}}. 
\end{equation*}
Since $\varphi(\cdot, x(t+h))$ is increasing and $\varphi(\lambda(t+h), x(t+h)) = \theta^{1/(p-1)}$, we have $\lambda(t+h) \geq \frac{\lambda(t)}{(1+\omega(h))^{p-1}}$.  Equivalently, we have
\begin{equation*}
\liminf_{h \rightarrow 0^+} \tfrac{\lambda(t+h) - \lambda(t)}{h} \geq - \lim_{h \rightarrow 0^+} \tfrac{\lambda(t)}{(1+\omega(h))^{p-1}} \cdot \tfrac{(1+\omega(h))^{p-1} - 1}{h}. 
\end{equation*}
By the definition of $\omega(h)$ and using the continuity of $x(\cdot)$, we have $\omega(h) \rightarrow 0$ as $h \rightarrow 0^+$. Since $p \geq 2$ is an integer, we have
\begin{equation*}
(1+\omega(h))^{p-1} - 1 = \sum_{i=1}^{p-1} \tfrac{(p-1)!}{i!(p-1-i)!}(\omega(h))^i = \omega(h)\left(p-1 + \sum_{i=1}^{p-2} \tfrac{(p-1)!}{(i+1)!(p-2-i)!}(\omega(h))^i\right). 
\end{equation*}
Further, we have 
\begin{equation*}
\tfrac{\omega(h)}{h} = \left(\tfrac{\lambda(t)}{\theta}\right)^{\frac{1}{p-1}}\tfrac{\|x(t+h) - x(t) - h\dot{x}(t)\|}{h} \rightarrow 0, \quad \textnormal{as } h \rightarrow 0^+. 
\end{equation*}
Putting these pieces together yields
\begin{equation*}
\tfrac{\lambda(t)}{(1+\omega(h))^{p-1}} \rightarrow \lambda(t), \quad \tfrac{(1+\omega(h))^{p-1} - 1}{h} \rightarrow 0, \quad \textnormal{as } h \rightarrow 0^+. 
\end{equation*}
Therefore, we conclude that $\dot{\lambda}(t) \geq 0$ for almost all $t \in [0, t_0)$ by achieving
\begin{equation*}
\liminf_{h \rightarrow 0^+} \tfrac{\lambda(t+h) - \lambda(t)}{h} \geq 0. 
\end{equation*}
This completes the proof.

\section{Proof of Lemma~\ref{Lemma:Lyapunov-descent}}
By the definition, we have
\begin{equation*}
\tfrac{d\ECal(t)}{dt} =  \langle\dot{x}(t), x(t) - z\rangle. 
\end{equation*}
In addition,  Eq.~\eqref{sys:general} implies that $\dot{x}(t) = - x(t) + (I + \lambda(t)A)^{-1}x(t)$.  Then, we have
\begin{equation}\label{inequality:Lyapunov-descent-first}
\tfrac{d\ECal(t)}{dt} = -\|x(t) - (I + \lambda(t)A)^{-1}x(t)\|^2 - \langle x(t) - (I + \lambda(t)A)^{-1}x(t), (I + \lambda(t)A)^{-1}x(t) - z\rangle.  
\end{equation}
Letting $y(t) = (I + \lambda(t)A)^{-1}x(t)$, we have $\frac{1}{\lambda(t)}(x(t) - y(t)) \in Ay(t)$. Since $z \in A^{-1}(0)$, we have $0 \in Az$. By the monotonicity of $A$, we have
\begin{equation*}
\tfrac{1}{\lambda(t)}\langle x(t) - y(t), y(t) - z\rangle \geq 0.  
\end{equation*}
Using $\lambda(t) > 0$ and the definition of $y(t)$, we have
\begin{equation}\label{inequality:Lyapunov-descent-second}
\langle x(t) - (I + \lambda(t)A)^{-1}x(t), (I + \lambda(t)A)^{-1}x(t) - z\rangle \geq 0. 
\end{equation}
Plugging Eq.~\eqref{inequality:Lyapunov-descent-second} into Eq.~\eqref{inequality:Lyapunov-descent-first} yields the desired inequality. 

\section{Proof of Lemma~\ref{Lemma:CAF-descent}}
It suffices to prove the first inequality in Eq.~\eqref{inequality:CAF-descent-main} which implies the other results. Indeed, we have
\begin{eqnarray}\label{inequality:CAF-descent-first}
\lefteqn{\ECal_k - \ECal_{k+1} = \langle x_k - x_{k+1},  x_{k+1} - z\rangle + \tfrac{1}{2}\|x_{k+1} - x_k\|^2} \\
& = & \langle x_k - x_{k+1},  y_{k+1} - z\rangle + \langle x_k - x_{k+1},  x_{k+1} - y_{k+1}\rangle + \tfrac{1}{2}\|x_{k+1} - x_k\|^2 \nonumber \\
& = & \underbrace{\langle x_k - x_{k+1},  y_{k+1} - z\rangle}_{\textbf{I}} + \tfrac{1}{2}\left(\underbrace{\|x_k - y_{k+1}\|^2 - \|x_{k+1} - y_{k+1}\|^2}_{\textbf{II}}\right). \nonumber 
\end{eqnarray}
Using the update $x_{k+1} = x_k - \lambda_{k+1}v_{k+1}$ and letting $v \in Az$, we have
\begin{equation*}
\textbf{I} = \lambda_{k+1} \langle v_{k+1},  y_{k+1} - z\rangle = \lambda_{k+1} \langle v_{k+1} - v,  y_{k+1} - z\rangle + \lambda_{k+1} \langle v,  y_{k+1} - z\rangle. 
\end{equation*}
Using $v_{k+1} \in A^{\epsilon_{k+1}}y_{k+1}$ and Eq.~\eqref{def:enlargement}, we have $\langle v_{k+1} - v,  y_{k+1} - z\rangle \geq -\epsilon_{k+1}$.  This implies
\begin{equation}\label{inequality:CAF-descent-second}
\textbf{I} \geq \lambda_{k+1} \langle v,  y_{k+1} - z\rangle - \lambda_{k+1}\epsilon_{k+1}. 
\end{equation}
Since $x_{k+1} = x_k - \lambda_{k+1}v_{k+1}$ and $\|\lambda_{k+1}v_{k+1} + y_{k+1} - x_k\|^2 + 2\lambda_{k+1}\epsilon_{k+1} \leq \sigma^2\|y_{k+1} - x_k\|^2$, we have
\begin{equation}\label{inequality:CAF-descent-third}
\textbf{II} = \|x_k - y_{k+1}\|^2 - \|\lambda_{k+1}v_{k+1} + y_{k+1} - x_k\|^2 \geq (1 - \sigma^2)\|x_k - y_{k+1}\|^2 + 2\lambda_{k+1}\epsilon_{k+1}. 
\end{equation}
Plugging Eq.~\eqref{inequality:CAF-descent-second} and Eq.~\eqref{inequality:CAF-descent-third} into Eq.~\eqref{inequality:CAF-descent-first}, we have
\begin{equation*}
\ECal_k - \ECal_{k+1} \geq \lambda_{k+1} \langle v,  y_{k+1} - z\rangle + \tfrac{1 - \sigma^2}{2}\|x_k - y_{k+1}\|^2,
\end{equation*}
which implies the desired inequality. 

\section{Proof of Lemma~\ref{Lemma:CAF-error}}
By the convention $0/0=0$, we define $\tau_k = \max\{\frac{2\epsilon_k}{\sigma^2}, \frac{\lambda_k\|v_k\|^2}{(1+\sigma)^2}\}$ for every integer $k \geq 1$. Then, we have
\begin{eqnarray*}
2\lambda_k\epsilon_k & \leq & \sigma^2\|y_k - x_{k-1}\|^2, \\
\|\lambda_k v_k\| & \leq & \|\lambda_k v_k + y_k - x_{k-1}\| + \|y_k - x_{k-1}\| \leq (1+\sigma)\|y_k - x_{k-1}\|. 
\end{eqnarray*}
which implies that $\lambda_k \tau_k \leq \|y_k - x_{k-1}\|^2$ for every integer $k \geq 1$. This together with Lemma~\ref{Lemma:CAF-descent} yields
\begin{equation*}
\tfrac{\inf_{z^\star \in A^{-1}(0)}\|x_0 - z^\star\|^2}{1-\sigma^2} \geq \sum_{i=1}^k \|y_i - x_{i-1}\|^2 \geq \left(\inf_{1 \leq i \leq k} \tau_i\right)\left(\sum_{i=1}^k \lambda_i\right). 
\end{equation*}
Combining this inequality with the definition of $\tau_k$ yields the desired results. 

\section{Proof of Lemma~\ref{Lemma:CAF-control}}
For $p=1$, the large-step condition implies that $\lambda_k \geq \theta$ for all $k \geq 0$.  For $p \geq 2$, the large-step condition implies
\begin{equation*}
\sum_{i=1}^k (\lambda_i)^{-\frac{2}{p-1}}\theta^{\frac{2}{p-1}} \leq \sum_{i=1}^k (\lambda_i)^{-\frac{2}{p-1}}(\lambda_i\|x_{i-1} - y_i\|^{p-1})^{\frac{2}{p-1}} = \sum_{i=1}^k \|x_{i-1} - y_i\|^2 \overset{\text{Lemma}~\ref{Lemma:CAF-descent}}{\leq} \tfrac{1}{1-\sigma^2}\left(\inf_{z^\star \in A^{-1}(0)}\|x_0 - z^\star\|^2\right). 
\end{equation*}
By the H\"{o}lder inequality, we have
\begin{equation*}
\sum_{i=1}^k 1 = \sum_{i=1}^k \left(\tfrac{1}{(\lambda_i)^{\frac{2}{p-1}}}\right)^{\frac{p-1}{p+1}} (\lambda_i)^{\frac{2}{p+1}} \leq \left(\sum_{i=1}^k \tfrac{1}{(\lambda_i)^{\frac{2}{p-1}}}\right)^{\frac{p-1}{p+1}}\left(\sum_{i=1}^k \lambda_i\right)^{\frac{2}{p+1}}.
\end{equation*}
For the ease of presentation, we define $C = \frac{1}{(1-\sigma^2)}\theta^{-\frac{2}{p-1}}(\inf_{z^\star \in A^{-1}(0)}\|x_0 - z^\star\|^2)$. Putting these pieces together yields
\begin{equation*}
k \leq C^{\frac{p-1}{p+1}}\left(\sum_{i=1}^k \lambda_i\right)^{\frac{2}{p+1}},
\end{equation*}
which implies
\begin{equation*}
\sum_{i=1}^k \lambda_i \geq \left(\tfrac{1}{C}\right)^{\frac{p-1}{2}} k^{\frac{p+1}{2}}.   
\end{equation*}
This completes the proof. 

\end{document}